\documentclass{amsart}
\usepackage[utf8]{inputenc}

\usepackage{skull}
\usepackage{phaistos}
\usepackage[margin=1in]{geometry}
\usepackage[T1]{fontenc}
\usepackage{csquotes}
\usepackage{dsfont}
\usepackage{hyperref}
\usepackage{url}
\usepackage[final]{microtype}
\usepackage{lmodern}
\usepackage{amsthm}
\usepackage{amssymb}
\usepackage{amsmath}
\usepackage{mathrsfs}
\usepackage{marvosym}
\usepackage{enumerate}
\usepackage{caption}
\usepackage{ragged2e}
\usepackage{tikz-cd} 
\usepackage{xypic} 
\usepackage{stmaryrd}
\usepackage[
   backend=biber,
   style=alphabetic, 
   sorting=nyt, 
   maxbibnames=99 
   ]{biblatex}
\renewbibmacro{in:}{}
\newcommand{\ds}{\displaystyle}
\newcommand{\Q}{\mathbb Q}
\newcommand{\Z}{\mathbb Z}
\newcommand{\R}{\mathbb R}
\newcommand{\C}{\mathbb C}
\newcommand{\A}{\mathbb A}

\renewcommand{\P}{\mathbb P}
\renewcommand{\cL}{\mathcal L}
\newcommand{\cM}{\mathcal M}

\newcommand{\cN}{\mathcal N}

\newcommand{\cO}{\mathcal O}

\newcommand{\cV}{\mathcal V}
\newcommand{\cX}{\mathcal X}
\newcommand{\cY}{\mathcal{Y}}
\newcommand{\cE}{\mathcal{E}}

\newcommand{\p}{\mathbf{p}}

\newcommand{\mathtext}[1]{\expandafter\def\csname#1\endcsname{\operatorname{#1}}}

\newcommand{\inv}{^{-1}}
\newcommand{\G}{\mathbb G}
\newcommand{\barEll}[1]{\overline{\operatorname{Ell}}_{#1}}
\newcommand{\naive}{\operatorname{naive}}

\mathtext{nm}
\mathtext{val}
\mathtext{Gal}
\mathtext{Aut}
\mathtext{edd}
\mathtext{Spec}
\mathtext{falt}
\mathtext{Mor}
\mathtext{Hom}
\mathtext{Im}
\mathtext{sm}
\mathtext{SL}
\mathtext{Ell}
\mathtext{PGL}
\mathtext{Ker}
\mathtext{Isom}
\mathtext{lcm}
\newcommand{\Ht}{\operatorname{ht}}
\newcommand{\Height}{\operatorname{Ht}}

\newcommand{\lesim}{\substack{< \\ \sim}}

\newtheorem{theorem}{Theorem}[section]

\newtheorem{proposition}[theorem]{Proposition}
\newtheorem{lemma}[theorem]{Lemma}
\newtheorem{claim}[theorem]{Claim}

\theoremstyle{definition}
\newtheorem{definition}{Definition}

\newtheorem{notation}{Notation}

\newtheorem{remark}{Remark}[section]
\newtheorem{example}{Example}[section]
\theoremstyle{remark}

\title{Counting elliptic curves with a rational $N$-isogeny for small $N$}
\author{Brandon Boggess and Soumya Sankar }

\begin{document}

\maketitle

\begin{abstract}
    We count the number of rational elliptic curves of bounded naive height that have a rational $N$-isogeny, for $N \in \{2,3,4,5,6,8,9,12,16,18\}$. For some $N$, this is done by generalizing a method of Harron and Snowden. For the remaining cases, we use the framework of Ellenberg, Satriano and Zureick-Brown, in which the naive height of an elliptic curve is the height of the corresponding point on a moduli stack.
\end{abstract}

\section{Introduction}


Let $E$ be an elliptic curve over $\Q$. An isogeny $\phi : E \rightarrow E^{'}$ between two elliptic curves is said to be \emph{cyclic of degree} $N$ if $\Ker (\phi) (\bar{\Q}) \cong \Z/N\Z$. Further, it is said to be \emph{rational} if $\Ker (\phi)$ is stable under the action of the absolute Galois group, $G_{\Q}$. A natural question one can ask is, how many elliptic curves over $\Q$ have a rational cyclic $N$-isogeny? Henceforth, we will omit the adjective `cyclic', since these are the only types of isogenies we will consider. It is classically known that for $N \le 10$ and $N=12,13,16, 18, 25$, there are infinitely many such elliptic curves. Thus we order them by naive height. An elliptic curve $E$ over $\Q$ has a unique minimal Weierstrass equation $y^2 = x^3 + Ax + B$ where $A,B\in \Z$ and $\gcd(A^3,B^2)$ is not divisible by any 12th power. Define the \emph{naive height} of $E$ to be $\Ht(E) = \max\{|A|^3, |B|^2\}$.

\begin{notation}
For two functions $f,g: \R \rightarrow \R$, we say that $f(X) \asymp g(X)$ if there exist positive constants $K_1$ and $K_2$ such that $K_1g(X) \le f(X) \le K_2g(X)$. For a positive real number $X$ and positive integer $N$, define
    $$
    \cN(N,X) = \# \{E/\Q \mid \Ht(E) < X, E \text{ has a rational } N \text{-isogeny} \}.
    $$
We are interested in finding a function $h_N(X)$ such that $\cN(N,X) \asymp h_N(X)$ for any real $X>0$. Note that $h_N(X)$ describes the rate of growth of $\cN(N,X)$, rather than being an asymptotic. In this paper we will often call it the \emph{asymptotic growth rate} of the the function $\cN(N,X).$
\end{notation}

\begin{theorem}
\label{theorem:introduction:maintheorem}
    Maintaining the notation above, we have the following values of $h_N(X).$
     \begin{table}[h]
     \centering
        \begin{tabular}{|c|c||c| c|}
        \hline
             $N$ & $h_N(X)$ & $N$ & $h_N(X)$ \\ \hline
             $2$ & $X^{1/2}$ & $8$ & $X^{1/6}\log(X)$  \\ \hline
             $3$ & $X^{1/2}$  & $9$ & $X^{1/6} \log(X)$ \\ \hline
             $4$ & $X^{1/3}$  & $12$ & $X^{1/6}$\\ \hline
             $5$ & $X^{1/6}{(\log(X))^2}$ & $16$ & $X^{1/6}$  \\ \hline
             $6$ & $X^{1/6} \log(X)$  & $18$ & $X^{1/6}$ \\ \hline
        \end{tabular}
        \caption{Values of $h_N(X)$, ordered by naive height}
        \label{table:maintheorem1}
        \end{table}
\end{theorem}

This result is motivated by work of Harron and Snowden in \cite{HS17}. In their paper they ask, for a given group $G$ from Mazur's list in \cite[Theorem 2]{Mazur78} , how many elliptic curves have $E(\Q)_{tors} \cong G$? They define
 $$
    \frac{1}{d(G)} = \lim_{X \rightarrow \infty} \frac{\log N_G(X)}{\log X},
$$
and compute $d(G)$ for each group in Mazur's list. Our counting results are a generalization of those in their paper, as explained later in this section.

\begin{remark}
    Some of these counts are not new, but we are able to give new proofs for them. Counting elliptic curves with a rational isogeny of degree 2 is equivalent to counting elliptic curves with a rational 2-torsion point. This is covered in \cite{HS17}. The case of $N=3$ was recently worked out by Pizzo, Pomerance and Voight \cite{PPV19}. The case $N=4$ was completed by Pomerance and Schaefer in \cite{PS20}, building off of work by Cullinan, Keeney and Voight in \cite{CKV20}. The case $N=4$ also follows from recent work of Bruin and Najman in \cite{BN20}, which we say more about in Remark \ref{remark:BruinNajman}. We include these cases in this paper since our methods are different, and the case $N=3$ serves as a good example for the purpose of exposition. An interesting phenomenon that occurs in this case is that of an `accumulating subvariety,' which occurs in a variety of questions related to the original Batyrev-Manin conjecture for schemes. It turns out that the $X^{1/2}$ contribution comes from elliptic curves with $j$-invariant 0, while the rest only contribute an $X^{1/3}\log(X).$ In a later section, we explain the absence of the cases $N=7,10,13$ and $25$. 
\end{remark}

\subsection{Counting rational points on stacks}

Let $\cX_0(N)$ be the compactification of the modular curve parametrizing pairs $(E,C)$, where $E$ is an elliptic curve and $C$ is a subgroup of $E$ isomorphic to $\Z/N\Z$. We can rephrase our problem of finding $h_N(X)$ in terms of counting rational points on $\cX_0(N)$. Each pair $(E,C)$ has an non-trivial automorphism. Thus $\cX_0(N)$ has generic inertia stack $B\mu_2$ and is not a scheme. For $N \le 10$ and $N=12,13,15,16,18$ and $25$, the coarse space of $\cX_0(N)$ can be identified with $\P^1$ via hauptmoduln. The difficulty in counting rational points on the stack $\cX_0(N)$ arises from the fact that there are many rational points with the same hauptmodul, parametrized by quadratic twists. \\

Let $\cX$ be a modular curve that is a scheme and let $E: y^2 = x^3 + Ax + B$ be a rational point on it. Then $\max \{|A|^3, |B|^2 \}$ is, up to a constant, the height with respect to the twelfth power of the Hodge bundle on $\cX$. So the natural question one might ask is, does the naive height also come from geometry in the case of $\cX_0(N)$? A positive answer to this question can be deduced from a forthcoming paper of Ellenberg, Satriano and Zureick-Brown (\cite{ESZB19}). In this paper, the authors establish a theory of heights on stacks. In particular, their height on $\cX_0(N)$ with respect to the 12th power of the Hodge bundle coincides with naive height. We use this geometric interpretation of naive height in \S \ref{section:countingheightsonstacks} to counts points on certain modular curves.

\subsection{Outline of Proof of Theorem \ref{theorem:introduction:maintheorem}}
\label{subsection:outlineofproof}
We use two main methods in the proof of the main theorem. For $N=3,4,6,8,9,12,16,18$, we generalize the methods of Harron and Snowden in \cite{HS17} to count elliptic curves in families. The idea of Harron and Snowden is to use an equation for the universal family of elliptic curves with a given torsion subgroup to reduce to a problem in analytic number theory -- counting polynomials over $\Q$ whose coefficients satisfy certain conditions. Counting elliptic curves with rational torsion isomorphic to $\Z/N\Z$ can be rephrased in terms of counting points on the modular curves $\cX_1(N)$ (see \S 2.1). For $5\le N\le 10$ and $N=12$, $\cX_1(N)$ is in fact a scheme isomorphic to $\P^1_{\Q}$, and so there exists a universal family $\mathcal E\to \cX_1(N)$ with a model $y^2=x^3+f(t)x+g(t)$ over $\Q(t)$. An elliptic curve $y^2=x^3+Ax+B$ over $\Q$ has a rational $N$-torsion point if and only if there exist $u,t\in\Q$ such that $A=u^4f(t)$ and $B=u^6g(t)$, and this is what is exploited by Harron and Snowden. The cases $N = 3, 4$ are handled by a slight modification of these methods.\\

The problem with applying this idea to counting elliptic curves with an $N$-isogeny is that $\cX_0(N)$ is not a scheme or even a stacky curve (in the sense of \cite{VZB15}). As such, there is no nice universal family at our disposal. Instead, the proof of Theorem \ref{theorem:introduction:maintheorem} involves reducing to a case where the framework of \cite{HS17} can be applied. We show that given an elliptic curve $E$ with an $N$-isogeny, there exists a twist $E^{\chi}$ of $E$ that can be interpreted as a rational point on either a scheme or at worst, a stacky curve. This curve, which we construct in \S 2,  is in fact a double cover of $\cX_0(N)$. Thus we reduce to counting quadratic twists within a framework very similar to the one used by Harron and Snowden in the $\cX_1(2)$ case. We note here that the technical counting theorems in \cite{HS17} are not enough to give us the results that we need. We thus prove a generalization of one of their theorems in \S 3, below.\\

For $N=2$ and $5$ this strategy does not work, since the cover of $\cX_0(N)$ that we construct is still a $\mu_2$ gerbe over its rigidification (in fact for $N=2$, it is equal to $\cX_0(N)$). Instead, we use the interpretation of naive height as the height with respect to the 12th power of the Hodge bundle. This allows us (in \S \ref{section:countingheightsonstacks}) to use different sections of the Hodge bundle to calculate the height of a rational point on $\cX_0(N)$, giving us a height that only differs from the naive height by a constant. We do this for $N=2,3,4,6,8$ and $9$ as well, thus giving a different proof for the asymptotic growth rates in these cases.

\begin{remark}
\label{remark:BruinNajman}
In recent work (\!\!\cite{BN20}), Bruin and Najman use the structure of $\cX_0(2)$ and $\cX_0(4)$ as weighted projective lines to obtain $h_2(X) = X^{1/2}$ and $h_4(X) = X^{1/3}$. This asymptotic holds over number fields as well. Although their proof is different from ours, the underlying idea of utilizing the stacky nature of these modular curves is the same.  
\end{remark}

\subsection{Acknowledgements}

We would like to thank Jordan Ellenberg for his valuable help and support. We would also like to thank David Zureick-Brown, John Voight and Jeremy Rouse for many helpful conversations and comments. We are also grateful to Andrew Snowden, Peter Bruin and Filip Najman for comments on the early draft of this paper. The second author would also like to thank Brandon Alberts and Libby Taylor.

\section{Preliminaries}

\subsection{Modular curves}

We start with some background and notation for modular curves. Most of this can be found in any standard textbook on modular curves. We recall this primarily in order to introduce the notation we will use for the rest of the paper. Let $S$ be a scheme. An \emph{elliptic curve over $S$} is a map $\pi:E\to S$, together with a section $z:S\to E$, such that every fiber of $\pi$ is a smooth projective genus 1 curve. 
Let $N$ be a positive integer. Let $\mathcal{Y}_0(N)$ denote the modular curve such that for a $\Z[\frac{1}{6N}]$-scheme $S$,
$$
\mathcal{Y}_0(N)(S) = \{(E/S, C/S) \mid C \cong_{S} \Z/N\Z \}
$$
where $E/S$ is an elliptic curve over $S$, $C$ is a sub-group scheme of $E$ defined over $S$, and the pair is taken up to isomorphism. Let $\cX_0(N)$ denote the compactification of $\mathcal{Y}_0(N)$ (in the sense of Deligne and Mumford). Every point of this moduli space possesses the extra automorphism $-1$, and so $\cX_0(N)$ is a stack with generic inertia stack $\mu_2$.\\

Let $\cY_1(N)$ denote the curve whose points are given by:
$$
\cY_1(N)(S) = \{(E/S,P/S) \mid N\cdot P = 0 \}
$$
where $E/S$ is an elliptic curve over $S$, $P \in E(S)$ is a point of order $N$, and the pair is taken up to isomorphism. Let $\cX_1(N)$ denote the Deligne-Mumford compactification of $\cY_1(N)$. For $N \ge 5$, $\cX_1(N)$ is a scheme. There is a natural map $\Phi_N : \cX_1(N) \rightarrow \cX_0(N)$ which sends $(E,P)$ to $(E, \langle P \rangle)$, where $ \langle P \rangle$ denotes the subgroup of $E$ generated by $P$. We remark here that the cusps of modular curves also have a moduli interpretation. They paramterize generalized elliptic curves with $\Gamma_0(N)$ or $\Gamma_1(N)$ structures. For a more detailed exposition on these, we refer the reader to \cite{DR73} or \cite{Con07}. A short summary can be found in Appendix \ref{appendix}.

\begin{definition}
    Let $\cM$ denote any modular curve. For any point $S \rightarrow \cM$, let $p: E \rightarrow S$ denote the corresponding elliptic curve. The \emph{Hodge bundle} $\lambda_{\cM}$ is the line bundle on $\cM$ such that $(\lambda_{\cM})_S = p_{*} \omega_{E/S}$. From the definition, one can see that if $\cM$ is a modular curve parametrizing elliptic curves with some level structure, then $\lambda_{\cM}$ is the pull back of $\lambda_{\cX(1)}$ along the forgetful map $\cM \rightarrow \cX(1)$. For ease of notation, we will omit the $\cM$ in $\lambda_{\cM}$ whenever the underlying modular curve is clear from context.
\end{definition}

Modular forms of weight $k$ and level $N$ are sections of the $k$-th power of the Hodge bundle on $\cX_0(N)$. The coefficients $A$ and $B$ in the Weierstass equation $y^2 = x^3 + Ax + B$ are, up to a scalar, the Eisenstein series $E_4$ and $E_6$ on $\cX(1)$ respectively. Thus $A^3$ and $B^2$ are sections of $\lambda^{\otimes 12}$ on $\cX_0(N)$. Thus counting elliptic curves of bounded naive height is the same as counting elliptic curves of bounded height with respect to $\lambda^{\otimes 12}$ on any modular curve that is a scheme (and as we shall see later, also moduli stacks).


\begin{definition}
\label{defn:hauptmoduln}
    Let $M$ denote the coarse space of a modular curve $\cM$. When $M \cong \P^1$, its function field is freely generated by a single element; this element is called a \emph{hauptmodul}. These hauptmoduln parametrize elliptic curves with a given level structure, and can be used to write equations for modular curves.
\end{definition}

\subsection{Rationally defined subgroups}\label{section:X_1/2}

In this subsection, we describe a degree two cover of $\cX_0(N)$ that we will use in our counting problem. To this end, let $N \ge 3$ 
and let $G = (\Z/N\Z)^{\times}$. Then $\Phi_N: \cX_1(N) \rightarrow \cX_0(N)$ is a branched $G$-cover of $\cX_0(N)$, with branch locus supported at irregular cusps and possibly points with $j=0, 1728$. Away from the branch locus, $G$ acts freely and transitively on the fibers of $\Phi_N$, by sending $a: (E,P) \mapsto (E, aP)$. Let $H$ be an index two subgroup of $G$. We denote by $\cX_{1/2}(N)$ the quotient $\cX_1(N)/H$. One can make sense of this quotient at the cusps by using the moduli interpretation of cusps as stated in \S 2.1; this construction is carried out in Appendix \ref{appendix}. We will denote by $\cY_{1/2}(N)$ the quotient $\cY_1(N)/H$.



\begin{remark} Before we proceed, we make some comments about the curves $\cX_{1/2}(N).$
\begin{enumerate}
    \item  The curve $\cX_{1/2}(N)$ is not a novel construction. It can be understood classically as the quotient of the upper half plane by an index 2 subgroup of $\Gamma_0(N)$. Further, we do not claim that $\cX_1(N)/H$ is a scheme. In fact it is a stack in many cases (see \S 4).
    \item The notation $\cX_{1/2}(N)$ might be misleading, since there is not always a unique index two subgroup of $(\Z/N\Z)^{\times}$. However, in our case we will only consider the $H$ for which $G/H$ is represented by $\{ +H, -H\}$. As an example, $(\Z/8\Z)^{\times}  \cong \Z/2\Z \times \Z/2\Z$. We will write this set as $\{1,3, 5, 7\}$. This has three index two subgroups: $H_1 = \{1,3 \}$, $H_2 = \{1, 5 \}$ and $H_3= \{1, 7\}$. The two cosets of $H_1$ are therefore $H_1 = \{1,3\}$ and $-H_1 =  \{5, 7\} $. Similarly for $H_2$. However, the two cosets of $H_3$ are $H_3 = \{1,7\} $ and $3H_3 = \{3, 5\}$. We will make it a point to \emph{not} pick $H_3$. The choice between $H_1$ or $H_2$ will not affect our final result. 
    \item  In the context of the remark above, we note that there are some values of $N$ (namely $N=5,10,13,25$) for which there is no choice of index 2 subgroup such that $G/H = \{\pm H\}$. For these $N$, while the construction of $\cX_{1/2}(N)$ still makes sense, it does not have the nice properties that we want (see Lemma \ref{lemma:twistsfactorization} and Proposition \ref{proposition:geometryofX1/2N}). Another way to rephrase the condition that $G/H = \{\pm H\}$ is in terms of the subgroup $\Gamma_0(N) \subset \SL_2(\Z)$. Consider the short exact sequence:
    $$
    1 \rightarrow \{\pm 1\} \rightarrow \Gamma_0(N) \rightarrow \P\Gamma_0(N) \rightarrow 1.
    $$
    For $N=5,10,13,25$, this sequence is non-split, while for the remaining $N$, it does split. This splitting enables us to construct a degree two cover of $\cX_0(N)$ without generic inertia. 
\end{enumerate}

\end{remark}

We now explain the significance of the curves $\cX_{1/2}(N)$. Most of what follows is well known (e.g., see \cite{RZB15}, \cite{GRSS}) but we recall them here for completeness. Let $E$ be an elliptic curve over $\Q$ with a rational $N$-isogeny. For notational convenience, we fix a Weierstrass form $y^2 = x^3 + Ax+ B$, with $A, B \in \Z$, for $E$. Fix an isomorphism of the kernel of the rational $N$-isogeny with $\Z/N\Z$. The Galois action of $G_{\Q}$ on the kernel defines a homomorphism:
$$
\chi : G_{\Q} \rightarrow (\Z/N\Z)^{\times}.
$$

For each $N \in \{3,4,6,7,8,9,12,16,18\}$, we can write $f: (\Z/N\Z)^{\times} \cong (\Z/2\Z) \times (\Z/m\Z)$ for some $m \in \Z$. This allows us to factor $\chi$ into two characters $\chi_1 : G_{\Q} \rightarrow \Z/2\Z $ and $\chi_2: G_{\Q} \rightarrow \Z/m\Z$. That is, we may write $\chi = \chi_1 \chi_2$ using the isomorphism $f$. Now, since $\chi_1$ is a quadratic character, it factors through a quadratic extension $K = \Q(\sqrt{d})$, with $d$ a squarefree integer. Let $E^{\chi_1}: dy^2 = x^3 + Ax + B$ denote the quadratic twist of $E$ over $K$.
    
 \begin{lemma}
   \label{lemma:twistsfactorization}
    Maintaining the above notation,
       $E^{\chi_1}$ has a rational $N$-torsion subgroup on which $G_{\Q}$ acts via $\chi_2$. That is, the Galois action on this $N$-torsion subgroup factors as:
       \begin{center}
           \begin{tikzcd}
                G \arrow[rr] \arrow[dr, "\chi_2"]&  &(\Z/N\Z)^{\times}\\
                 &\Z/m\Z \arrow[ur, hook] & \\
           \end{tikzcd}
       \end{center}
   \end{lemma}
   \begin{proof}
    Let $C$ denote the kernel of the rational $N$ isogeny of $E$. Let $\phi : E \rightarrow E^{\chi_1}$ denote the isomorphism of elliptic curves defined over $\Q(\sqrt{d})$. For $P \in C$ and $\sigma \in G_{\Q}$, $P^{\sigma}=\chi(\sigma)P$ by assumption. Further, by the definition of a twist,
    $$
    \phi(P)^{\sigma} = \chi_1(\overline{\sigma}) \phi(P^{\sigma})
    $$
    where $\overline{\sigma}$ is the image of $\sigma$ in $\Z/2\Z$. Since $\chi_1$ is quadratic, $\chi_1 \chi = \chi_2$. It follows that $G_{\Q}$ acts on $\phi(C)$ via $\chi_2$.
   \end{proof}

  We have thus proved the following for $N \in \{3,4,6,7,8,9,12,16,18\}$.
   
   \begin{proposition}
   \label{prop:MainPropositionXHalfN}
        Fix an appropriate index 2 subgroup $H \subset (\Z/N\Z)^{\times}$ and consider the corresponding curve $\cX_{1/2}(N)$. Let $(E,C) \in \cX_0(N)(\Q)$. Then there exists a unique $d \in \Z$ squarefree, such that the corresponding twist $(E^{\chi_1},\phi(C))$ satisfies:
        \begin{enumerate}
            \item $(\phi(C))^{\times}$ has an index two subgroup $H_C$ defined over $\Q$, and therefore
            \item $(E,H_C), (E, -H_C) \in \cX_{1/2}(N)(\Q)$. 
        \end{enumerate}
   \end{proposition}
   
   \begin{proof}
       This follows from combining the interpretation of $\cX_{1/2}(N)$ as a fiberwise quotient of $\cX_1(N)$ with Lemma \ref{lemma:twistsfactorization}.
   \end{proof}
   
       A nice example of Proposition \ref{prop:MainPropositionXHalfN} is in the cases $N=3,4,6$, where $(\Z/N\Z)^{\times} \cong \Z/2\Z$. In these cases $\cX_{1/2}(N) = \cX_1(N)$. For these values of $N$, Proposition \ref{prop:MainPropositionXHalfN} says that if $E$ has a rational $N$-isogeny then there exists a quadratic twist of $E$ that has a rational $N$ torsion point.


\subsection{Automorphisms and universal families}
\label{subsection:auts}


 In this section, we briefly recall the relation between automorphisms and the existence of universal families. For more details, we refer the reader to \cite{KM85}, Chapter 4 and Appendix A.4. Let $F$ be a functor on the category ${\Ell}$ of elliptic curves over a ring $R$. Let $\tilde{F}$ denote the corresponding functor on the category of $R$-schemes sending an $R$-scheme $S$ to isomorphism classes of pairs $(E/S, \alpha)$, where $E$ is an elliptic curve over $S$ and $\alpha \in F(E/S)$ is an `$F$-level structure'. The functor $F$ (resp. $\tilde{F}$) is \emph{representable} if there exists a universal elliptic curve $\cE$ over a scheme $\cM$ (resp. a scheme $\cM$) such that $F(E/S) = \Hom(E/S, \cE/\cM)$ (resp. $\tilde{F}(S) = \Hom(S,\cM)$). Note that the representability of $F$ guarantees the existence of $\cM$, and therefore implies the representability of $\tilde{F}$. The functor $F$ is said to \emph{rigid} if for any $E/S \in {\Ell}$, and any $\alpha \in F(E/S)$, the pair $(E/S, \alpha)$ has no non-trivial automorphisms. In general, if $F$ is representable, then $F$ is rigid. The following proposition tells us when the converse is true:
    
    \begin{proposition}[\cite{KM85}, 4.7.0]
    \label{prop:rigidandrep}
        Suppose that for every elliptic curve $E/S$, the functor on the category of schemes over $S$ defined by
        $$
        T \mapsto F(E_T/T)
        $$
        is representable by a scheme. Suppose further that $F$ is affine over ${\Ell}$, that is, the morphism $F_{E/S} \rightarrow S$ is affine. Then $F$ is representable if and only $F$ is rigid.
    \end{proposition}
    
    In this paper, we will be interested in the functors of points corresponding to $\cX_0(N)$, $\cX_1(N)$ and the intermediate quotient $\cX_{1/2}(N)$. To see that in these cases, the two hypotheses of Proposition \ref{prop:rigidandrep} are satisfied, we refer the reader to \cite{KM85}. Thus we may move freely between the existence of universal families and rigidity.

\subsection{Counting lattice points in a region}

In this section, we state a theorem of Davenport on a Lipschitz principle (\cite{Dav51}). Let $\mathcal{R}$ be a closed and bounded region in $\R^n$. Suppose $\mathcal{R}$ satifies the following two conditions:
\begin{enumerate}
    \item Any line parallel to one of the coordinate axes intersects $\mathcal{R}$ in a set that is a union of at most $h$ intervals.
    \item The same is true (with $n$ replaced by $m$) for any of the $m$-dimensional regions obtained by projecting $\mathcal{R}$ down to an $m$-dimensional coordinate axis ($1 \le m \le n-1$).
\end{enumerate}

Let $V(\mathcal{R})$ be the volume of the region $\mathcal{R}$ and $N(\mathcal{R})$ the number of lattice points in it. Then, the following theorem holds.

\begin{theorem}[\cite{Dav51}]
\label{theorem:davenport}
    For $\mathcal{R}$ satisfying 1 and 2, 
    $$
    |N(\mathcal{R}) - V(\mathcal{R})| \le \sum_{m=0}^{n-1} h^{n-m}V_m,
    $$
    where $V_m$ is the sum of the ($m$-dimensional) volumes of the $m$-dimensional projections of $\mathcal{R}$ and $V_0=1$.
\end{theorem}

We will use this theorem repeatedly in the next section.

\section{Counting quadratic twists in families}

From \S 2, we see that in order to count elliptic curves in $\cX_0(N)(\Q)$ with respect to naive height, we must count elliptic curves for which there exists a quadratic twist that gives a rational point on $\cX_{1/2}(N)(\Q)$. In this section we state and prove the counting results that will enable us to do so.

\begin{proposition}[\cite{HS17}, Theorem 4.1]
\label{prop:HarronSnowdenTwistTheorem}
Let $f,g \in \Q[t]$ be coprime polynomials of degrees $r$ and $s$ respectively. Let $\max\{r,s\} > 0$ and let $m$ and $n$ be coprime integers such that
$$
\max \left\{\frac{r}{2}, \frac{s}{3}\right\} = \frac{n}{m}.
$$
Assume that either $n=1$ or $m=1$. Let $S(X)$ be the set of pairs $(A,B) \in \Z^2$ such that
\begin{itemize}
    \item $4A^3 + 27B^2 \neq 0$
    \item $\gcd(A^3, B^2)$ not divisible by a 12th power,
    \item $|A |< X^{1/3}$ and $|B| < X^{1/2}$, 
    \item $\exists u , t \in \Q$ such that $A = u^2f(t)$ and $B = u^3 g(t)$.
\end{itemize}

Define 
$$
k(x) = \begin{cases} X^{(m+1)/6n} \quad & m+1> n\\
X^{1/6} \log(X) \quad & m+1 = n\\
X^{1/6} \quad & m+1 <n
\end{cases}
$$
Then, 

$$
S(X) \asymp k(x).
$$
\end{proposition}

As we will see in Section 4, this theorem is not enough for all the cases that we are interested in. For $N=3$, the condition: `either $n=1$ or $m=1$' is not satisfied. We will thus prove a generalization of this proposition.

\begin{remark}
We note here that we do not prove the most general version of Theorem \ref{theorem:MainTheorem2Twists} possible, since we do not need it. It might be an interesting exercise in analytic number theory to prove such a version, independent of the interpretation of counting points on a moduli space.
\end{remark}


 \begin{theorem}
    \label{theorem:MainTheorem2Twists}
    Let $f,g \in \Q[t]$ be coprime polynomials of degrees $r$ and $s$ respectively. Let $\max\{r,s\} > 0$ and let $m$ and $n$ be coprime integers such that $$\max \left\{\frac{r}{2}, \frac{s}{3} \right\} = \frac{n}{m}.$$
    Suppose that $n, m \neq 1$. Define
    $$
    h = \left\lfloor \frac{n(m-1)}{m} \right\rfloor
    $$
    
    and $w = \max \{\frac{3h}{s}, \frac{2h}{r} \}$. Suppose further that
    \begin{itemize}
        \item $m+1 >n$, 
        \item $\frac{m+1}{n} - (w+1) = -1,$ and
        \item $\min \{3rm - 6h, 2sm - 6h \} \le 6$.
    \end{itemize}

    Let $S(X)$ be the set of pairs $(A,B) \in \Z^2$ such that
    \begin{itemize}
        \item $4A^3 + 27B^2 \neq 0$
        \item $\gcd(A^3, B^2)$ not divisible by a 12th power,
        \item $|A| < X^{1/3}$ and $|B| < X^{1/2}$,
        \item $\exists u , t \in \Q$ such that $A = u^2f(t)$ and $B = u^3 g(t)$.
    \end{itemize}

Then, 

$$
S(X) \asymp  X^{(m+1)/6n}\log(X).
$$
\end{theorem}

\begin{remark}
    Note that the hypotheses on $m,n, r$ and $s$ make it so that there aren't many choices of these variables that satisfy all the hypotheses together. The degree conditions for $\cX_0(3)$, which give $m=3, n=2, h=1$ and $w=2$, are perhaps the only moduli problem of interest that that satisfy these. However, stating the theorem in this manner instead of using numbers makes the method less opaque and more amenable to generalization.
\end{remark}

\subsubsection{Proof of Theorem \ref{theorem:MainTheorem2Twists}}

The proof of this theorem closely follows that in \cite{HS17}. We provide the key parts of the proof here for the sake of completeness. We prove the upper bound and the lower bound in two separate sections. For the reader's convenience, we outline each proof first.

\begin{notation}
    For any two real valued functions $h(X)$ and $k(X)$, we say that $h(X) \lesim k(X)$ if there is a positive constant $C$ such that $h(X) \le  C k(X)$.
\end{notation}

\textbf{Upper bound.} Our goal is to reduce the problem of counting pairs in $S(X)$ to the problem of counting tuples of integers in a bounded region, perhaps with some divisibility conditions. Let $S_1(X)$ be the set of $u$, $t$ such that $(u^2f(t), u^3g(t)) \in S(X).$ Counting $S_1(X)$ gives an upper bound for $S(X)$. We will express $u$ and $t$ as $qc^{-1}db^n$ and $ab^{-m}$ respectively for some integers $a$, $b$, $c$ and $d$ and some rational number $q$. Lemmas \ref{lemma:countinglemma1}, \ref{lemma:quadtwistlemma1} and \ref{lemma:quadtwistlemma2} enable us to do this. The next key observation is that there are only finitely many possibilities for $q$. Thus for the kind of upper bound that we are looking for, we can count 4-tuples of integers in a particular region. Lemma \ref{lemma:quadtwistlemma3} gives the bounds for such a region. Lemma \ref{lemma:quadtwistlemmafinalupperbound} outlines what divisibility conditions these integers must satisfy, and also calculates the number of such tuples.
     \hfill $\square$\\
     
\begin{lemma}[\cite{HS17}, Lemma 2.2]
    \label{lemma:countinglemma1}
        For each place $p$ of $\Q$, there is a constant $c_{p}>0$ such that for each $t \in \Q$:
        $$
        \max(|f(t)|_{p}, |g(t)|_{p}) \ge c_{p}.
        $$
        Furthermore, we can take $c_{p} = 1$ for all sufficiently large $p$. 
\end{lemma}

Let $S_1(X)$ be the set of $u$, $t$ such that $(u^2f(t), u^3g(t)) \in S(X).$ 

\begin{lemma}
\label{lemma:quadtwistlemma1}
         For each prime $p$ there is a constant $C_p$ such that for all $(u,t) \in S_1(X)$, we have:
        \begin{align}
            \val_p(u) = \epsilon{'} + \begin{cases} \lceil -\frac{n}{m}\val_p(t) \rceil & \val_p(t)<0\\
            0 & \val_p(t) \ge 0
            \end{cases}
        \end{align}
        for some $|\epsilon{'}| \le C_p$. Moreover, we can take $C_p = 1$ for all $p$ sufficiently large.
    \end{lemma}
    
    \begin{proof}
        The proof of this lemma closely follows that of Lemma 2.3 in \cite{HS17}. Fix a prime $p$. Since $A$ and $B$ must be integral, we have that:
        \begin{align}
            \val_p(A) &= 2\val_p(u) + \val_p(f(t)) \ge 0\\
            \val_p(B) &= 3\val_p(u) + \val_p(g(t)) \ge 0
        \end{align}
    
        Thus,
        \begin{align}
        \label{equation:valuationofulowerbound}
        \val_p(u) \ge \max \left( \lceil -\frac{1}{2} \val_p(f(t)) \rceil, \lceil -\frac{1}{3} \val_p(g(t)) \rceil  \right) =: K.
        \end{align}
        
        Note that if $\val_p(u) \ge 2+K$, then by replacing $u$ by $p^2u$ we see that $p^{12}\mid \gcd(|A|^3, B^2)$. Thus we must have $K \le \val_p(u) \le K+1$. The rest of the proof goes exactly like in \cite{HS17}. Suppose $\val_p(t)< 0$. Pick $K_1$ such that $|\val_p(f(t)) - r\val_p(t)| < K_1$ and $|\val_p(g(t)) - s\val_p(t)|< K_1$ for all such $t$. Note that $K_1$ can depend on $p$ and is 0 for large enough $p$. Then,
       $$
       K = \epsilon + \max \left( \lceil -\frac{r}{2} \val_p(t) \rceil, \lceil -\frac{s}{3} \val_p(t) \rceil  \right) = \epsilon + \left \lceil \frac{-n}{m} \val_p(t)  \right\rceil 
       $$
       where $|\epsilon| < K_2$ for some $K_2$. Thus we have:
       $$
       \epsilon + \left\lceil \frac{-n}{m} \val_p(t)  \right\rceil \le \val_p(u) \le \epsilon + \left\lceil \frac{-n}{m} \val_p(t)  \right\rceil + 1
       $$
       for $\val_p(t) < 0$. \\
       
       Now consider the case when $\val_p(t) \ge 0$. By Lemma \ref{lemma:countinglemma1}, there exists $K_3$ such that $\min( \val_p(f(t)), \val_p(g(t)) \le K_3 $. Further, $K_3 = 0$ for $p\gg0$. Thus $-\val_p(u) \le K_4$ for some constant $K_4$. Since $\val_p(t) \ge 0$, there is a $K_5$ such that $\val_p(f(t)) \ge K_5$ and $\val_p(g(t)) \ge K_5$ for all such $t$. This gives a lower bound on $-\val_p(u)$, appealing again to (\ref{equation:valuationofulowerbound}). Thus there is a constant $K_7$ such that $|\val_p(u)| \le K_7$. We remark here to avoid confusion that all the $K_i$'s are constant with respect to $t$ and $u$, but do depend on $p$, $f$ and $g$.\\
       
       This gives us first part of the lemma. For the second part of the lemma, we need only take, as in \cite{HS17}, $p\gg0$ such that: (1) the coefficients of $f$ and $g$ are $p$-integral, (2) the leading coefficients of $f$ and $g$ are $p$-units and (3) the constant $c_p$ in Lemma \ref{lemma:countinglemma1} can be taken to be 1. Since $K \le \val_p(u) \le K+1$, we can only get $C_p=1$ for $p \gg0$.
       
    \end{proof}
    
 The next step is to prove an analogue of Lemma 2.4 in \cite{HS17}. This will enable us to reduce our problem to that of counting lattice points in a region. We start with some notation. Recall that $w = \max\{3h/s, 2h/r \}$. For a given pair of positive integers $(a,b)$, we say a prime $p$ satisfies $(*)$ if:
        $$
        p|b \implies p^{w}| a
        $$
           
\begin{lemma}
    \label{lemma:quadtwistlemma2}
     Suppose $(u,t) \in S_1(X)$. There is a \emph{finite} set $Q \subset \Q^{\times}$ (independent of $u$ and $t$) such that: we can write $t = ab^{-m}$ and $u = qc^{-1}db^n$, where:
        \begin{enumerate}
            \item $a,b \in \Z$, with $b>0$,
            \item $\gcd(a,b^m)$ is $m$-th power free,
            \item $d$ is a squarefree integer,
            \item $q \in Q$, and
            \item $c \in \Z$ such that $\val_p(c) \le h$ for all $p$ and $\val_p(c)>0$ if and only if $p$ satisfies $(*)$.
        \end{enumerate}
\end{lemma}

\begin{proof}
        Given $t \in \Q$, one can always write $t = ab^{-m}$ satisfying (1) and (2). Pick any such representation. We now analyze $ub^{-n}$ and show that $\val_p(ub^{-n})$ must satisfy the required constraints. For convenience we will fix $N_0$ to be an integer such that $C_p=1$ for $p \ge N_0$. Such an $N_0$ exists by Lemma \ref{lemma:quadtwistlemma1}.\\
        
        We divide the set of all primes into two groups: $p|b$ and $p \nmid b$. If $p \nmid b$, then $\val_p(t) \ge 0$ and by Lemma \ref{lemma:quadtwistlemma1}, we have $|\val_p(ub^{-n})| \le C_p$. If $p|b$, then we write $-\val_p(t) = m\val_p(b) - k$, where $0\le k < m$. Therefore, by Lemma \ref{lemma:quadtwistlemma1} again, we have:
        $$
        \val_p(u) = \epsilon{'} + n\val_p(b) + \lceil \frac{-n}{m} k\rceil.
        $$
        Therefore for any $p$, we have $-C_p - h \le \val_p(ub^{-n}) \le C_p$.\\
        
        If $p \le N_0$, we have no control over $C_p$, but we know that there are finitely many possibilities for the $N_0$-smooth part of $ub^{-n}$, since $|\val_p(ub^{-n})| \le C_p + h$ (here,  $N_0$-smooth means the part of the numerator or denominator that is divisible only by primes less than or equal to $N_0$). For $p \ge N_0$, we have:
        $$
        |\val_p{(ub^{-n}})| \le 1 \text{ if } p\nmid b 
        $$
        and
        $$
        -1-h \le \val_p{(ub^{-n})} \le 1 \text{ if } p|b.
        $$
        
        In the case that $p\nmid b$ and $p\ge N_0$, we see that $\val_p(t) \ge 0$. Further, in the proof of the previous lemma, $N_0$ is picked so that for $p \ge N_0$, the $\val_p(f(t)) \ge 0$ and $\val_p(g(t)) \ge 0$, with at least one of them being an equality. In particular, this implies that $\val_p(ub^{-n}) \ge 0$ for such $p$. Similarly, for $p|b$ ($\val_p(t) < 0$) and $p \ge N_0$, we can take $\epsilon{'} = 0$. Thus, $\val_p(ub^{-n}) = \lceil \frac{-n}{m}k\rceil$ or $ \lceil \frac{-n}{m}k\rceil +1$. Therefore $\val_p(ub^{-n}) \ge -h$.\\
        
        We factor the $p \ge N_0$ part of $ub^{-n}$ as $c^{-1}d$, where $\val_p(d) \neq 0$ iff either $p\nmid b$  or $\val_p(ub^{-n}) = 1$ if $p|b$. Further, in these cases, we set $\val_p(d) = \val_p(ub^{-n})$. The previous paragraph shows that $d$ is a squarefree integer and that $\val_p(c) \le h$.  \\
       
       We now explain the condition $(*)$. This comes from the fact that $A$ and $B$ are required to be integers. For any $p$:
       \begin{align*}
        \val_p(u^2f(t)) &= 2\val_p(qc^{-1}b^n) + \val_p(f(t))\\ 
        &= 2\val_p(q) - 2\val_p(c) + 2m(n/m - r/2)\val_p(b) + r\val_p(a) + K_1\\
        &\ge 2\val_p(q) + K_1 + r\val_p(a) - 2\val_p(c)
       \end{align*}
       
       where $K_1$ is a positive constant that can be taken to be 0 for $p\gg0$. Similarly, for $B$, we get that: $\val_p(u^3g(t)) = 3\val_p(q) + K_1 + s\val_p(a) - 3\val_p(c).$ Since $q$ is $N_0$-smooth, for $p$ large enough, the condition of integrality of $A$ and $B$ translates directly to condition $(*)$. Further, since we are only interested in an upper bound for the asymptotic growth, not imposing conditions on say, $2\val_p(q) + K_1$ for small $p$ causes us no harm. 
\end{proof}

Now consider $(u,t) \in S_1(X)$ and write them as in Lemma \ref{lemma:quadtwistlemma2}. The fact that $\max \{|A|^3, B^2\} < X$ implies bounds for $a$, $b$, $c$ and $d$, which we now find.
    
    \begin{lemma}
            \label{lemma:quadtwistlemma3}
            Let $(u,t) \in S_1(X)$. Represent $u=qc^{-1}db^n$ and $t=ab^{-m}$ as in Lemma \ref{lemma:quadtwistlemma2}. Then,
            $$
            |a| \lesim  X^{m/6n} c^{m/n} d^{-m/n} \qquad \text{and} \qquad |b| \lesim  X^{1/6n} c^{1/n} d^{-1/n}.
            $$
        \end{lemma}
        
        \begin{proof}
              If $A = u^2f(t)$ and $B = u^3g(t)$, the bound $\max(|A|^3, B^2) < X$ translates to:
    $$
    |u| \max(|f(t)|^{1/2}, |g(t)|^{1/3} ) < X^{1/6}.
    $$
    Let $K$ be the positive constant such that $\max(|f(t)|^{1/2}, |g(t)|^{1/3} ) > K$ for all $t$. Thus:
    $
    |u| \le K^{-1}X^{1/6}.
    $
    Let $M_2 = K^{-1}(\max_{q \in Q}|q|^{-1})$. Thus, we have that:
    $$
    |c^{-1}db^{n}| < M_2 X^{1/6}
    $$
    
    i.e. $|b| < M_2 X^{1/6n} c^{1/n} d^{-1/n}$.\\
    
    We now turn to bounding $a (=tb^{m})$. Suppose $t<1$. Then, by the above bound for $b$, we have $|a| < M_2 X^{m/6n} c^{m/n} d^{-m/n}$. If $t \ge 1$, then we can find a constant $M>0$ such that $M^2|t|^r \le |f(t)|$ and $M^3|t|^s \le |g(t)|$. Thus we have:
        $$
        X^{1/6} > |u|\max(|f(t)|^{1/2}, |g(t)|^{1/3}) > M|u|\max(|t|^{r/2}, |t|^{s/3}) = M|ut^{n/m}|.
        $$
        Now, $|ut^{n/m}| = |qc^{-1}da^{n/m}|$ and so, $|c^{-1}da^{n/m}| < M^{-1}(\max_{q \in Q} |q|^{-1})X^{1/6} $. Thus, we see that:
        $$
        |a| \lesim  X^{m/6n} c^{m/n} d^{-m/n}.
        $$
        \end{proof}


 \begin{lemma}
 \label{lemma:quadtwistlemmafinalupperbound}
    Under the hypotheses of Theorem \ref{theorem:MainTheorem2Twists}, $|S_1(X)| \lesim X^{m+1/6n}\log(X)$.
 \end{lemma}
        
\begin{proof}
     Fix a $c>1$. Let $S_1(X;c)$ denote the set of all $(a,b,d) \in \Z^3$ such that:
        \begin{enumerate}
            \item $|ad^{m/n}| < X^{m/6n}c^{m/n}$,
            \item $|bd^{1/n}| < X^{1/6n}c^{1/n}$,
            \item $p|c \iff p|b, p^{w}|a$, and $\val_p(c) \le h$ for all $p$.
        \end{enumerate}
        
        Let $T(X;d,c) = \{(a,b) \mid (a,b,d) \in S_1(X;c) \}$. By standard analytic number theory and Theorem \ref{theorem:davenport}, it follows that
        $$|T(X;d,c)| = \frac{1}{c^{w+1}} X^{(m+1)/6n}d^{-(m+1)/n}c^{(m+1)/n} + {O}\left(\frac{1}{c^{w+1}} X^{m/6n}d^{-m/n}c^{m/n} \right).
        $$
 
        
        Thus we have,
        \begin{align*}
            |S_1(X;c)| &= \sum_{d< X^{1/6}} |T(X;d,c)|\\
            &= \sum_{d< X^{1/6}} c^{(m+1)/n - (w+1)} X^{(m+1)/6n}d^{-(m+1)/n}  + O\left(\sum_{d<X^{1/6}} \frac{1}{c^{w+1}} X^{m/6n}d^{-m/n}c^{m/n} \right).
        \end{align*}
        
        We will only consider the case $m+1>n$. Further, since $m \neq n$, the error term above just becomes
        $$
        O\left( \frac{1}{c^{w+1}} X^{m/6n}c^{m/n} \right).
        $$
        
        Therefore we have,
        $$
        |S_1(X;c)| \lesim c^{(m+1)/n - (w+1)} X^{(m+1)/6n}.
        $$
        
        Summing over $h+1$-th power-free $c$, with $c < X^{\alpha}$ (for any $\alpha$), since $\frac{m+1}{n} - (w+1) = -1$, we have
        $$
        |S_1(X)| \lesim X^{(m+1)/6n}\log(X).
        $$
        
\end{proof}

\textbf{Lower Bound.} The outline of the proof of the lower bound is as follows: we know that if $(u,t) \in S_1(X)$, then $u$ and $t$ have expressions as in Lemma \ref{lemma:quadtwistlemma2}. Instead of counting all of these, we only count ones of the form $u= c^{-1}b^n$ and $t=ab^{-m}$, where $a, b$ and $c$ are within appropriate bounds. Let $S_2(X)$ be the set of such triples $(a,b,c)$. There is a map $S_2(X) \rightarrow S(X)$, and the bulk of the proof is in showing that this map has bounded fibers. We first form another intermediary set, which we call $S_3(X)$. We then describe maps $S_2(X) \rightarrow S_3(X) \rightarrow S(X)$, and bound the fibers of these maps. This will enable us to find a lower bound for $S(X)$ by finding one for $S_2(X)$ instead.  \hfill $\square$ \\
 
 Since we only need a lower bound, observe that by changing $u$ to $Mu$ for large enough $M$, we can assume that $f(t), g(t) \in \Z[t]$. For a triple $(a,b,c) \in \Z^{3}$, set $u = c^{-1}b^{n}$ and $t = ab^{-m}$. Let $A = u^2f(t)$ and $B = u^3g(t)$. Fix some constant $\kappa >0$. Define $S_2(X)$ to be: the set of triples $(a,b,c) \in \Z^3$ such that:
            \begin{itemize}
                \item $c = \ds \prod_{p|b, p^w|a} p^h$,
                \item $0<b< \kappa X^{1/6n}c^{1/n}$, $|a| < \kappa X^{m/6n}c^{m/n}$, $\gcd(a,b^m)$ is $m$-th power free 
                \item $4A^3 + 27B^2 \neq 0$ (where $A$ and $B$ are as defined above). 
            \end{itemize}
        
           Note that if $(a,b) \in S_2(X)$, then for a suitable value of $\kappa$, we get $(A,B) \in S(X)$, since $|A| = |u^2f(t)| \lesim |c^{-2} a^r b^{2n - mr}| \lesim X^{1/3}$, and similarly for $B$.\\

     \textbf{Notation: } Define $S_3(X) \subset \Z^2$ to be the set of $(A,B) \in \Z^2$ coming from $S_2(X).$ We then have a map from $S_3(X) \rightarrow S(X)$ sending $(A,B) \mapsto (A/d^4, B/d^6)$ where $d^{12} \mid\mid \gcd(A^3, B^2)$. Stratify $S_2(X)$ by sets $S_2(X;c)$ of pairs $(a,b)$ such that $\prod_{p|b, p^w|a} p^h = c$. Define $S_3(X;c)$ as the pairs $(A, B)$ coming from $(a,b) \in S_2(X;c)$.\\
        
    The following lemma will help us bound the fibers of the map $S_3(X) \rightarrow S(X)$. 
        
    \begin{lemma}
    \label{lemma:quadtwistlowerbound1}
        There exists a non-zero integer $D$ (depending only on $f$ and $g$) with the following property: if $(a,b,c) \in S_2(X)$, then $\gcd(A^3, B^2)$ can be factored as $(M_D) \beta$ such that $M_D$ divides $D$ and $p|\beta \implies p|b$.
    \end{lemma}
    
     \begin{proof}
         We follow the same method of proof as in Harron and Snowden. Let $(a,b) \in S_2(X;c)$ and let $p$ be a prime. Let $M_1$ be a constant such that $|3\val_p(f(t)) - 3r\val_p(t)|< M_1$ and $|2\val_p(g(t)) - 2s\val_p(t)|< M_1$ for all $t \in \Q$ with $\val_p(t)<0$. Let $M_2$ be the constant for which $\min \{3\val_p(f(t)), 2\val_p(g(t))\} \le M_2 $ for all $t \in \Q$ with $\val_p(t) \ge 0$. Note that $\max\{ M_1, M_2\}$ is 0 for $p\gg0$ (specifically, $p \ge N_0$, as defined in Lemma \ref{lemma:quadtwistlemma3}). 
         
         
         Now, consider the case where $\val_p(t) <0$. In particular, $p|b$. Let $\val_p(b) = k$ and let $\val_p(a) = l (<m)$. We then have:
        \begin{align*}
        \val_p(A^3) = \begin{cases}6mk\left( \frac{n}{m} - \frac{r}{2}\right) + 3rl - 6h + \epsilon & p|c \\ 6mk\left( \frac{n}{m} - \frac{r}{2}\right) + 3rl + \epsilon & p \nmid c  \end{cases}\\
        \val_p(B^2) = \begin{cases}6mk \left( \frac{n}{m} - \frac{s}{3} \right) + 2sl - 6h+ \delta &p|c \\ 6mk \left( \frac{n}{m} - \frac{s}{3} \right) + 2sl + \delta & p \nmid c \end{cases}
        \end{align*}
        
        where $|\epsilon| < M_1$ and $|\delta|< M_1$. Let $M_0 = \min \{3rm, 2sm\} $.
        Let 
        $$
        e_p = \begin{cases} \max\{ M_1 + M_0, M_2\} & p \le N_0 \\ M_0 & p \ge N_0 \end{cases}
        $$
        and take $D = \prod_{p \le N_0} p^{e_p}$. This proves the lemma.
        \end{proof}
        
          \begin{remark}
             We find that $D$ is $N_0$-smooth and $\beta$ consists of $p|b$ for $p \ge N_0$. It is crucial that $D, M_1, M_2$ and $M_0$ do not depend on $(a,b,c)$ in any way. They only depend on $f$ and $g$.
        \end{remark}

     We now use this lemma to to bound the fibers of $S_3(X) \rightarrow S(X)$ in our case of interest, namely when:
    $$
    \min \{3rm - 6h, 2sm - 6h \} \le 6.
    $$
    We will call this assumption $(**)$.
        
      \begin{lemma}
    \label{lemma:quadtwistlowerbound2}
        There exists a constant $N$ such that the size of the fibers of $S_3(X) \rightarrow S(X)$ is bounded by $N$.
    \end{lemma}
    
    \begin{proof}
        The fiber over a point $(A^{'},B^{'}) \in S(X)$ is in bijection 
        with the set $\{d \in \Z \mid (d^4A^{'}, d^6B^{'}) \in S_3(X) \}$. Thus for any $(A,B) \in S_3(X)$, the size of the fiber above the pair is bounded above by the number of 12th powers dividing $\gcd(|A|^3, B^2)$. We show that this is exactly the number of 12th powers dividing $D$ from Lemma \ref{lemma:quadtwistlowerbound1}, i.e. no 12th powers divide $\beta$.\\
        
        Consider a prime $p \ge N_0$. We claim that $p^{12}$ cannot divide $\gcd(|A|^3, B^2)$. If $p|b$ and $p|c$, then this follows from assumption $(**)$. If $p\nmid b$, then since $K_2 =0$, $p$ doesn't divide $\gcd(|A|^3, B^2 )$. If $p|b$ and $p \nmid c$, then by definition of $c$, we must have that $p^w \nmid a$. Since $h=1$, this forces $l \le 1$ in Lemma \ref{lemma:quadtwistlowerbound1}. Since assumption $(**)$ implies $\min \{3r, 2s \} < 12$, we are done.
    \end{proof}
        
     The rest of the proof follows by the exact argument as that in Harron and Snowden, which we recall below.
  
  \begin{lemma}
        \label{lemma:quadtwistlowerbound3}
        There exists a constant $M$ such that every fiber of the map $S_2(X) \rightarrow S_3(X)$ has size bounded by $M$.
  \end{lemma}
  
  \begin{proof}
      Fix any $(A,B) \in S_3(X)$. An element in fiber of the map $S_2(X) \rightarrow S_3(X)$ above $(A,B)$ is of the form $(a,b,c) \in \Z^3$ with $A=(c^{-1}b^{n})^2 f(ab^{-m})$ and $B = (c^{-1}b^{n})^3 g(ab^{-m})$, and $c = \prod_{p|b, p^w|a}p^h$. Set $x = cb^{-n}$ and $y = ab^{-m}$. Then an element $(a,b,c)$ in the fiber satisfies the equations: 
      $$
      Ax^2 = f(y) \qquad Bx^3 = g(y).
      $$
      These can be thought of as defining curves in $\P^2$, that intersect transversally, since $f$ and $g$ are coprime. Thus by Bezout's theorem, the maximum number of solutions is bounded above by: $M = \max(2,r) \max(3,s)$. 
  \end{proof}
  
        It only remains to bound the size of $S_2(X)$. Now, $S_2(X) = \coprod_{c} S_2(X;c)$ and the size of $S_2(X;c)$ is precisely:
            $$
            \frac{1}{c^{w+1}} c^{(m+1)/n} X^{(m+1)/6n} + O \left( \frac{1}{c^{w+1}} c^{m/n} X^{m/6n} \right),
            $$
            
            where the error term comes from \ref{theorem:davenport}.
            Summing over $c$ gives us: $X^{(m+1)/6n}\log(X) \lesim S_2(X).$ 
            


 
\section{Proof of Theorem \ref{theorem:introduction:maintheorem} for $N\ne 2, 5$}
\label{section:proofofmaintheorem}
The proof of Theorem \ref{theorem:introduction:maintheorem} for the cases $N\ne2, 5$ involves applying the appropriate theorems from \S 3. We start off with a list of the modular curves of genus zero that we consider, and their geometric descriptions. Note that when we say that a curve has $n$ \emph{stacky points}, we are talking about $n$ stacky geometric points (\cite[\href{https://stacks.math.columbia.edu/tag/04XE}{Tag 04XE}]{stacks-project} ). For modular curves, this can be thought of as referring to $n$ distinct values of the corresponding hauptmoduln (Definition \ref{defn:hauptmoduln}) or cusps. Recall that we use the term `stacky curve' as defined in \cite{VZB15} to mean curves that have a trivial generic inertia stack.
 
 
 \begin{proposition}
 \label{proposition:geometryofX1/2N}
    For any $N \in \Z_{>0}$, consider the curve $\cX_{1/2}(N)$ constructed in \S \ref{section:X_1/2}. Then:
    \begin{enumerate}
        \item If $N = 3$, then $\cX_{1/2}(N) = \cX_1(3)$, which is a stacky curve with (geometrically) one stacky point corresponding to the elliptic curves with $j$-invariant 0.
        \item If $N=4$, then $\cX_{1/2}(N) = \cX_1(4)$, which is a stacky curve whose only stacky point is at the irregular cusp.
        \item If $N=7$, then $\cX_{1/2}(N)$ is a stacky curve with two stacky points whose hauptmoduln are defined over $K = \Q(\sqrt{-3})$ and are conjugate over $\Q$.
        \item If $N= 6,8,9,12,16,18$, then $\cX_{1/2}(N)$ is a scheme.
        \item If $N=5,10,13,25$, then $\cX_{1/2}(N)$ has generic inertia stack $B\mu_2$.
    \end{enumerate}
 \end{proposition}
 \begin{proof}
    This follows from the construction of $\cX_{1/2}(N)$, by analysing the automorphisms of its points and applying Proposition \ref{prop:rigidandrep}. For $N=3,4$ and $6$, this is classical, as in each of these cases $\cX_{1/2}(N) = \cX_1(N)$. We demonstrate the cases $N=5$, $7$ and $8$, and leave the rest to the reader.\\
    
    Consider the map $\cX_{1/2}(N) \rightarrow \cX_0(N).$ Since any point in $\cX_{1/2}(N)$ lies in some geometric fiber of this map, it is enough to analyse automorphisms of points in each fiber. For any point $(E,C) \in \cX_0(N)$, choose an isomorphism $C \cong \Z/N\Z$ and thus $\Aut(C) \cong (\Z/N\Z)^{\times}$. Let $P$ be a generator for $C$.\\
    
    For $N=7$, the fiber above a point $(E,C)$ contains the points $(E, \{P,2P,4P \})$ and $(E, \{-P,-2P, -4P \})$. If $E$ has non-zero $j$-invariant, then the only extra automorphism of the pair $(E,C)$ is $[-1]$ and thus the points in the fiber do not have any extra automorphisms. Recall that $\cX_0(7)$ has exactly two elliptic points, both with $j$-invariant 0. In order to find these two points, we first note that the universal family over $\cY_1(7)$ is
    \[
        y^2+(1+v-v^2)xy+(v^2-v^3)=x^3+(v^2-v^3)x^2,
    \]
    with torsion point $(0,0)$ \cite[Table 3]{Kub76}. The $j$-invariant of the universal family is
    \[
        \frac{(v^6-11v^5+30v^4-15v^3-10v^2+5v+1)^3(v^2-v+1)^3}{(v-1)^7v^7(v^3-8v^2+5v+1)}.
    \]
    This gives exactly eight values of $v$ producing a curve of $j$-invariant 0. An explicit computation with the torsion point confirms that the stacky points correspond to the roots of $v^2-v+1$, which are defined over $\Q(\sqrt{-3})$.\\
    

    If $N=8$, then recall from \S 2 that the choice of index 2 subgroup of $(\Z/N\Z)^{\times}$ is not unique, and we choose one that works for us. That is, write $(\Z/8\Z)^{\times} =\{P,3P,-3P, P\}$ and say we chose the subgroup $\{P,3P\}$, so that fiber above $(E,C)$ consists of the points $(E,\{P,3P\})$ and $(E, \{-P,-3P\})$. Neither pair has extra automorphisms. 
    If $N=5$. Then $C = \{P,2P,-2P,-1P\}$, which has a unique index 2 subgroup: $\{P,-P\}$. Thus the fiber above $(E,C)$ has two points: $(E, \{P,-P\})$ and  $(E, \{2P, -2P\})$. Each of these points still has the automorphism $[-1]$. This proves the theorem for $N=5$.
 \end{proof}
    


    What this proposition tells us is that if $N \in \{3,4,6,7,8,9,12,16,18\},$ then there is an open sub-stack $\mathcal{U}$ of $\cX_{1/2}(N)$ that is isomorphic to a scheme. Therefore $\mathcal{U}(\Q)$ can be parametrized via the universal family over $\mathcal{U}$. For $N \in \{4,6,8,9,12,16,18\}$, the non-stacky locus contains $\cY_{1/2}(N)$, and thus there exist $f_N$ and $g_N \in \Q[t]$ coprime such that every elliptic curve arising from a rational point on $\cY_{1/2}(N)$ is isomorphic to one of the form:
    $$
     \cE_{N,t} : y^2 = x^3 + f_N(t)x + g_N(t).
    $$
    Thus, by Proposition \ref{prop:MainPropositionXHalfN}, we have the following:
    \begin{align*}
          \cN(N,X) &= \#\{E/ \Q \mid \Ht(E) < X, \text{ and } \exists d \in \Z, u, t \in \Q, \text{s.t. } E_d: y^2 = x^3 + u^4f_N(t)x + u^6g_N(t)\}\\
          &= \#\{E/ \Q \mid \Ht(E) < X, \text{ and } \exists u, t \in \Q, \text{s.t. } E: y^2 = x^3 + u^2f_N(t)x + u^3g_N(t)  \}.
    \end{align*}
  
    To find the asymptotic growth for $\cN(N,X)$ in these cases, we use Proposition \ref{prop:HarronSnowdenTwistTheorem} to find the value of $h_N(X)$, given in Table \ref{table:maintheoremproof} below.\\
    
    For $N=3$, the situation is slightly different. $\cX_{1/2}(3) = \cX_1(3)$ has one stacky point lying above the elliptic curve with $j$-invariant 0. Let $\Phi_3 : \cX_{1/2}(3) \rightarrow \cX(1)$ be the usual forgetful map. Set $Y = \cY_{1/2}(3) \setminus \phi_3^{-1}(\{ j=0\})$. Then, for a suitable embedding of $Y \hookrightarrow \mathbb{A}^1$, there is a universal family $\cE_{3,t}$ over $Y$ (e.g. see \cite{HS17}) given by:
    $$
    \cE_{3,t} : y^2 = x^3 + \left(2t - \frac{1}{3} \right)x + \left(t^2 - \frac{2}{3}t + \frac{2}{27} \right).
    $$
    
    Every elliptic curve with non zero $j$-invariant and a rational 3-torsion point is isomorphic to one of the above form for some $t \in \Q$. However, this family does not extend to a universal family over $t = 1/6$. Indeed $\cE_{3, 1/6}$ is given by $y^2 = x^3 - \frac{1}{108}$ and its torsion subgroup of order 3 is generated by the rational point: $(1/3, 1/6)$. 
    On the other hand, all curves $E^D : y^2 = x^3 + D^2$, $D \in \Q$ contain the rational 3 torsion point $(0,D)$ and have $j$-invariant $0$, but none of them is isomorphic to $\cE_{3,1/6}$ over $\Q$. For this reason, we separate our counting function into two pieces:
    $$
    \cN(3,X) = \cN(3,X)_{j=0} + \cN(3,X)_{j \neq 0}.
    $$
    
    By Theorem \ref{theorem:MainTheorem2Twists}, we have the following proposition:
    \begin{proposition}
        Maintaining the notation as above, 
        $$
        \cN(3,X)_{j \neq 0} \asymp X^{1/3}\log(X)
        $$
    \end{proposition}
    
    \noindent
    In order to find the asymptotics for $\cN(3,X)_{j=0}$, we observe the following: by Lemma 3.4 in \cite{HS17}, we know that any elliptic curve that has $j$-invariant 0, a rational 3 torsion point, but is not of the form $\cE_{3,t}$ for any $t \in \Q$, admits an equation of the form $y^2 = x^3 + D^2$, $D \in \Z$. Thus the curves missing from our count are those that are quadratic twists of these exceptional curves. That is, they are elliptic curves of the form:
    $$
    y^2 = x^3 + u^3t^2
    $$
    for some $u, t \in \Q$ with $u^3t^2$ integral and minimal. This is the same as counting elliptic curves $y^2 = x^3 + b$, with $b^2 < X$ and $b$ 6th power free. This number is just a constant times $X^{1/2}$.
        
    \begin{remark}
        Note that our result agrees with that in \cite{PPV19}. In fact the argument for $\cN(3,X)_{j=0}$ is exactly the same as in their paper, albeit stated slightly differently. 
    \end{remark}
    
     To complete the proof of the main theorem, for each $N$ we need only calculate $r, s, m$ and $n$ in the notation of Proposition \ref{prop:HarronSnowdenTwistTheorem} and Theorem \ref{theorem:MainTheorem2Twists}. In Table \ref{table:maintheoremproof}, we give the components required to compute $r$ and $s$ in each of the cases of interest (in the notation of the above theorems).

    \begin{table}[h]
    \centering
        \begin{tabular}{|c|c|c|c|c|c|c|}
        \hline
            $N$ & $r$ & $s$ & $m$ & $n$ & Reference &$h_N(X)$  \\ \hline
             3& 1&2 &3 & 2& \ref{theorem:MainTheorem2Twists} & ${X^{1/2}}$ \\ \hline
             4& 2&3 & 1&1 & \ref{prop:HarronSnowdenTwistTheorem} & $X^{1/3}$ \\ \hline
             6& 4& 6& 1&2 & \ref{prop:HarronSnowdenTwistTheorem} & $X^{1/6}\log(X)$ \\ \hline
             8& 4& 6&1 &2 & \ref{prop:HarronSnowdenTwistTheorem} & $X^{1/6}\log(X)$ \\ \hline
             9& 4&6 &1 &2 & \ref{prop:HarronSnowdenTwistTheorem} & $X^{1/6}\log(X)$ \\ \hline
             12& 8& 12& 1& 4& \ref{prop:HarronSnowdenTwistTheorem}& $X^{1/6}$ \\ \hline
             16& 8& 12& 1& 4& \ref{prop:HarronSnowdenTwistTheorem} & $X^{1/6}$ \\ \hline
             18&12 &18 &1 &6 & \ref{prop:HarronSnowdenTwistTheorem} & $X^{1/6}$\\ \hline
        \end{tabular}
        \caption{Values of invariants}
        \label{table:maintheoremproof}
    \end{table}
    
    \begin{remark}
        We now explain the reason for the omission of $\cX_0(7)$ from our asymptotics. Our general strategy of counting points on $\cX_0(7)$ by counting quadratic twists of points on $\cX_{1/2}(7)$ still makes sense. However, the universal family that we obtain for the subscheme of $\cX_{1/2}(7)$ is a little bit worse for counting. More precisely, let $Y$ denote the largest substack of $\cX_{1/2}(7)$ that is isomorphic to a scheme and doesn't contain any cusps. Then there exist $f$ and $g \in \Q[t]$ such that for any $E$ coming from $Y(\Q)$, $E$ is isomorphic to an elliptic curve of the form $y^2=x^3 + f(t)x+g(t)$ for some $t \in \Q$. However, $f$ and $g$ are not coprime. For instance, if $t$ was taken to be the hauptmoduln $ (\eta_1/\eta_7)^4$, then $f$ and $g$ would have a common factor of $t^2 +13t+49$. One might wonder if this might be resolved choosing $f$ and $g$ cleverly, but that is not the case. This is an artifact of $\cX_{1/2}(7)$ having two stacky points, neither of which is rational, which makes it impossible to move the lack of semistability to $\infty \in \P^1$. 
    \end{remark}

\section{Counting points of bounded height on stacks}
\label{section:countingheightsonstacks}
In this section, we prove Theorem \ref{theorem:introduction:maintheorem} for $N=2,3,4,5,6,8,9$ by using results from \cite{ESZB19}. 
As we have seen, one can define \emph{some height} on $\cX_0(N)$, namely the naive height. The question is does this height come from geometry? We know that this is true for modular curves that are schemes (see \S 2.1) -- the naive height is the height with respect to the twelfth power of the Hodge bundle. It follows from the work in \cite{ESZB19} that the same is true for moduli \emph{stacks} of elliptic curves, and we use their machinery to count the number of points of bounded height. Before we proceed, we must set some notation:

\begin{notation}
    Recall that we use $\Ht(E)$ for the naive height of a point $E$ on any modular curve. Let $\cX$ be a stack and $\cV$ a vector bundle on it. We will let $\Ht_{\cV}$ denote the logarithmic height with respect to $\cV$ as defined in \cite{ESZB19} and $\Height_{\cV}$ the multiplicative height corresponding to it. That is to say, $\Height_{\cV} = \exp(\Ht_{\cV})$.
\end{notation}

We will not define $\Ht_{\cV}$ here, but we will use the fact that if $\cV = \lambda^{\otimes 12}$ on $\cX_0(N)$, then for an elliptic curve $E$ corresponding to a rational point $x : \Spec \Q \rightarrow \cX_0(N)$, 
$
\log\Ht(E) = \Ht_{\cV}(x) + O(1)
$
(see Example \ref{example:hodgeheightonX_1} below). Thus our counting function satisfies 
    $$
    \cN(N,X) \asymp \# \{x \in \cX_0(N)(\Q) \mid \Height_{\lambda}^{12}(x) < X \}.
    $$

\subsection{Computing heights on stacks}

Throughout this subsection, $\cX$ will be a proper Artin stack over $\Spec \Z$ with finite diagonal. A $\Q$-rational point $x$ of $\cX$ is a map $x : \Spec \Q \rightarrow \cX$. Let $\cV$ be a vector bundle on $\cX$. Consider for a moment the special case where $\cX = X$, a proper scheme, and $\cV$ is an ample line bundle on it. 
When computing the height of a point on $X$, we use a power of $\cV$ to embed $X \hookrightarrow \P^n$ for some $n$, and then use the naive height of the image of the point on $\P^n$. This makes computations easier. For a stack, the analogue would be mapping it into weighted projective space. In \cite{ESZB19}, the authors show that this works. We recall the specific result below.\\



Consider the special case where $\cV$ is a metrized line bundle $\cL$ (see \cite{ESZB19} for precise definition). Suppose $s_1, s_2, \dotsc, s_k$ are sections of $\cL$. Then, $\cL$ is said to be \emph{generically globally generated by $s_1,\dotsc,s_k$} if the cokernel of the corresponding morphism 
$$
\cO_{\cX}^{\oplus k} \rightarrow \cL
$$
vanishes over the generic point of $\Spec \Z$. In particular, this implies that the cokernel is supported at finitely many places.

\begin{proposition}[\cite{ESZB19}, Proposition 2.27]
\label{prop:ESZBstackyheight}
Let $\cX$ be a stack over $\Spec \Z$, let $\cL$ be a line bundle on $\cX$ such that $\cL^{\otimes n}$ is generically globally generated by sections $s_1, s_2 \cdots s_k$. Let $x : \Spec \Q \rightarrow \cX$ and for each $i$, let $x_i = x^{*}(s_i)$ (after picking an identification of $x^{*}\cL$ with $\Q$). Scale $x_1,\dotsc,x_k$ so that each $x_i\in\Z$ and for every prime $p$, there is some $x_i$ such that $v_p(x_i)<n$. Then
$$
\Ht_{\cL}(x) = \frac{1}{n} \log\max_{i} \{|x_1|, |x_2| \ldots |x_k| \} + O_{\cX(\Q)}(1)
$$
where $|\cdot|$ is the usual archimedean absolute value.
\end{proposition}

Note here that we have only stated the version of the proposition that we require, i.e. for $\Spec \Q$ and $\Spec \Z$. A more general version of this proposition holds for other global fields. We will say that the tuple $(x_1, \ldots x_k) \in \Z^k$ is \emph{minimal} if it satisfies the last condition in the theorem: for each prime $p$, there is some $i \in \{ 1 \ldots k\}$ such that $p^{n}\nmid x_i$. 

\begin{example}
\label{example:hodgeheightonX_1}
Let $\cL = \lambda$, the Hodge bundle on $\cX(1)$. Then the global sections of $\lambda^{\otimes 12}$ are weight 12 modular forms, and it is a classical fact that the Eisenstein series $E^3_4,E^2_6$ generically globally generate $\lambda^{\otimes 12}$. An elliptic curve $E:y^2=x^3+Ax+B$ gives a $\Q$-point $x:\Spec\Q\to\cX(1)$. The assumption about scaling the sections corresponds to choosing a minimal Weierstrass equation for $E$. Proposition \ref{prop:ESZBstackyheight} then says that
\[
    \Ht_{\lambda}(x) = \frac{1}{12}\log\max\{|A|^3, |B|^2\} + O(1),
\]
which is, up to the constant $O(1)$, a twelfth of the logarithmic naive height of $E$. Thus, $\Height^{12}_{\lambda}(x)$ is a constant multiple of the naive height $\Ht(E)$ as defined in \S 1.
\end{example}

\subsection{The ring of modular forms of low level}

Since modular forms are sections of powers of the Hodge bundle, we will rely on the structure of the rings of modular forms of $\cX_0(N)$ quite heavily. This subsection summarizes part of the work of Hayato and Tomohiko in \cite{HT11}.

\begin{notation}
    Let $M_k(N)$ denote the space of modular forms for $\Gamma_0(N)$ of weight $k$. We let $M(N) = \bigoplus_k M_k(N)$ be the entire ring of modular forms for $\Gamma_0(N)$.
    \begin{itemize}
        \item $E_k$: classical Eisenstein series of weight $k$, normalized to have constant coefficient equal to 1. Note that $E_k \in M_k(1)$ for $k \ge 4$.
        \item For a modular form $f$ and an integer $h$, let $f^{(h)}(q) = f(q^h)$. 
        \item For any $N \ge 1$, let $C_N = \frac{1}{\gcd(N-1,24)} (NE^{(N)}_2 - E_2) \in M_2(N)$.
        \item For certain $d \in \Z_{>0}$, Hayato and Tomohiko define modular forms $\alpha_d$ and $\beta_d$. We refer the reader to \cite{HT11} for the precise definitions, since we do not use them. The crucial properties of these modular forms that we use are their weight, level and the fact that they have integral coefficients.
    \end{itemize}
\end{notation}

\begin{proposition}[\cite{HT11}, Theorems 1,2]
\label{prop:ringofmodularforms}
    Under the above notation, the rings of modular forms for $\Gamma_0(N)$ for $N \in \{2,3,4,5,6,8,9\}$ are as follows:
    \leavevmode
    \begin{center}
        \begin{tabular}{|c|c|c|}
        \hline
            $N$ & Degrees of generators & $M(N)$ \\
            \hline
            $2$ & $(2,4)$ & $\C[C_2, \alpha_2]$ \\
             \hline
             $3$ & $(2,4,6)$ & $\C[C_3, \alpha_3, \beta_3]/(O_3)$ \\
             \hline
             $4$& $(2,2)$ & $\C[C_2, C_4]$\\
             \hline
             $5$ & $(2,4,4)$ & $\C[C_5, \alpha_5,\beta_5]/(O_5)$
             \\ \hline
             $6$ & $(2,2,2)$ & $\C[C_3^{(2)}, \alpha_6, \beta_6 ]/(O_6)$\\
             \hline
             $8$ & $(2,2,2)$ & $\C[C_4^{(2)}, \alpha_4, \alpha_4^{(2)} ]/(O_8)$ \\
             \hline
             $9$ & $(2,2,2)$ & $\C[C_3, \alpha_9, \beta_9 ]/(O_9)$\\
             \hline
        \end{tabular}
        \captionof{table}{Rings of modular forms of low level}
        \label{table:ringsofmodularforms}
    \end{center}
    
    Here the $O_n$'s are explicit polynomials whose form we will mention later. 
\end{proposition}


    Since for each $N$, $M(N)$ is graded by weight, the degrees in Table \ref{table:ringsofmodularforms} refer to the weights in which the corresponding rings are generated. In what follows, we will use the structure of the ring of modular forms of the levels in Table \ref{table:ringsofmodularforms} to count points of bounded height. The reason we restrict to these cases is that for such $N$, the rings of modular forms are easier to handle. For some other $N$ (e.g. see \S 6), this method reduces to one of counting integral points on more complicated varieties. 

\subsection{Counting results}


Our counting results will be split into three parts: the first, for $N=2,4$ corresponds to the $N$ for which $M(N)$ is freely generated. The second part, is for $N=3,6,8,9$. These are the levels $N$ for which the corresponding $O_N$'s in Table \ref{table:ringsofmodularforms} have a similar form. The last part is for $N=5$, which has to be dealt with separately because $O_5$ has a starkly different form, and thus requires different counting techniques.

\begin{notation}
\label{notation:conditiondagger}
    Let $(a_1, \ldots a_k) \in \Z^k$, $\p = (p_1, \ldots p_k) \in \Z_{>0}^k$ be two tuples of integers. Let $n \in \Z_{>0}$ such that $\lcm(p_1,\ldots p_k) | n$. We will say that the pair, $((a_1, \ldots a_k), \p)$ satisfies condition $(\dagger)$ if for any prime $p$:
    $$
    p^n \nmid \gcd_i(|a_i|^{p_i}).
    $$
    Condition $(\dagger)$ reflects the minimality condition in Proposition \ref{prop:ESZBstackyheight}.
\end{notation}

\noindent
\textbf{The cases $N=2,4:$} From Table \ref{table:ringsofmodularforms}, we see that $M(2) \cong \C[x,y]_{(2,4)}$ and $M(4) \cong \C[x,y]_{(2,2)}$. For $N=2$, $\lambda^{\otimes 12}$ is globally generated by $C_2^6$ and and $\alpha_2^3$. Let $x : \Spec \Q \rightarrow \cX_0(2)$. Let $a = x^{*}(C_2)$ and $b = x^{*}(\alpha_2)$. Taking these to be in minimal form implies that $p^{12} \nmid \gcd(a^6,b^3)$. Then, by Proposition \ref{prop:ESZBstackyheight}, we see that:
$$
\Ht_{\lambda}(x) = \frac{1}{12}\log \max \{|a|^6, |b|^3 \} + O_{\cX_0(2)(\Q)}(1).
$$
Thus we have that:
\begin{align*}
    \cN(2,X) &:= \# \{x \in \cX_0(2)(\Q) \mid \Height_{\lambda}^{12}(x) < X \}\\
    &\asymp \# \{(a,b) \in \Z^2 \mid ((a,b),(6,3)) \text{ satisfies } (\dagger),  \max\{|a|^6,|b|^3\}<X \}. 
\end{align*}

By a similar argument, observing that $C_2^6$ and $C_4^6$ globally generate $\lambda^{12}$ on $\cX_0(4)$, we set $a = x^{*}(C_2)$ and $b = x^{*}(C_4)$. Thus:
$$
\cN(4,X) \asymp \#\{(a,b) \in \Z^2 \mid ((a,b),(6,6)) \text{ satisfies } (\dagger),  \max\{|a|^6,|b|^6\}<X \}.
$$

In each of these cases, our counting problem reduces to counting integers in a box with certain divisibility conditions. The set we need to count has the form $\{(a,b) \in \Z^2 \mid |a|<M, |b|<N, p^{12}\nmid \gcd(a^{p_1}, b^{p_2}) \}$ for some constants $p_1, p_2, M$ and $N$. The set $\{ (a,b) \in \Z^2 \mid |a|<M, |b|<N, \gcd(a,b) =1\} $ is always a subset of this set, and in particular, has size a constant multiple of $MN$. Thus the condition $(\dagger)$ does not affect the asymptotic growth rate.

\begin{proposition}
Maintaining the above notation, we have:
  \begin{align*}
      \cN(2,X) \asymp X^{1/2}\\
      \cN(4,X) \asymp X^{1/3}.
  \end{align*}
\end{proposition}

Note that this agrees with the asymptotics in \cite{HS17}, \cite{PS20} as well as the conclusions of Table \ref{table:maintheoremproof} in \S 4.\\ 

\noindent
\textbf{The cases $N=3,6,8,9$:} These cases are similar because of the similarity in form on the $O_N$'s in table \ref{table:ringsofmodularforms}. More precisely, we have from \cite{HT11}:
\begin{itemize}
    \item $O_3 = \alpha_3^2 - C_3 \beta_3$
    \item $O_6 = \alpha_6^2 - C_3^{(2)}\beta_6 $
    \item $O_8 = \alpha_4^2 - C_4^{(2)}\alpha_4^{(2)}$
    \item $O_9 = \alpha_9^2 - C_3\beta_9$
\end{itemize}

In order to deal with these cases uniformly, we must introduce some notation. For $(a,b,c) \in \Z^3$ and $\mathbf{p} = (p_a, p_b, p_c) \in \Z_{>0}^3$, define:
$$
\Height^{\mathbf{p}}(a,b,c) = \max \{|a|^{p_a}, |b|^{p_b}, |c|^{p_c}\}.
$$

Later, for each $N$, we will fix a choice of $\p$ that makes this height compatible with $\Height_{\lambda}^{12}$ on $\cX_0(N)$. We will be interested in the following counting functions: 
\begin{align*}
    \cN(\p, X) &:=
    \#\{(a,b,c) \in \Z^3  \mid \Height^{\p}(a,b,c) < X, b^2 = ac \},\\
    \cN(\p, X, \dagger) &:= \#\{(a,b,c) \in \Z^3  \mid \Height^{\p}(a,b,c) < X, b^2 = ac,\; \text{and} \; ((a,b,c),\p) \text{ satisfies }(\dagger) \}.
\end{align*}


\begin{lemma}
\label{lemma:stackycount3689}
    There is a positive constant $C$ that depends only on $\p$ and $n$ such that:
    $$\cN(\p, X) = C\, X^{1/p_b} \log(X) + X^{1/p_c} + O(X^{1/p_a}).$$
\end{lemma}

\begin{proof}
    We start by noting that we must have $|a| < X^{\frac{1}{p_a}}$, $|b| < X^{\frac{1}{p_b}}$ and $|c| < X^{\frac{1}{p_c}}$. Now suppose $a \neq 0$. Then:
    \begin{align*}
        \sum_{\substack{|a| < X^{\frac{1}{p_a}}\\ a \neq 0} } \sum_{|c| < X^{\frac{1}{p_c}}} \sum_{\substack{|b| < X^{\frac{1}{p_b}}\\ b^2 = ac}} 1
        &= \sum_{\substack{|a| < X^{\frac{1}{p_a}} \\ a \neq 0}} \sum_{|b| < X^{\frac{1}{p_b}}, a|b^2} 1 \\
        &= \sum_{\substack{|a| < X^{\frac{1}{p_a}} \\ a \neq 0 }} \left( \frac{X^{1/p_b}}{a} + O(1) \right)\\
        &= C. X^{1/p_b} \log(X) + O(X^{1/p_a}).
    \end{align*}
    One might worry here that the `error' term, $X^{1/p_a}$, is actually bigger than the main terms. However, for all of our cases $p_a \ge p_b, p_c$, so $X^{1/p_a}$ will indeed be an error term. If $a = 0$, then $b$ is necessarily $0$ too. Thus we are reduced to counting $\#\{c \in \Z \mid c < X^{1/p_c}\} = X^{1/p_c} + O(1)$.
\end{proof}

\begin{claim}
\label{claim:reducingclaim}
    $\cN(\p,X, \dagger)$ is a positive proportion of
 $\cN(\p, X)$.
\end{claim}

\begin{proof}
    Firstly, note that if $a=0$, then $\#\{c \in \Z \mid c \text{ is } p_c\text{th power free}, |c|<X^{1/p_c} \}$  is a positive proportion of $\#\{c \in \Z \mid |c|<X^{1/p_c} \}$. In particular these sizes differ by a factor of $\zeta(p_c)$. Now suppose $a\neq 0$. Then the set of triples satisfying $(\dagger)$ contains those for which $a$ is squarefree. The proof of Lemma \ref{lemma:stackycount3689} shows that the set of such triples has size a constant times $X^{1/p_b}\log(X)$ as well. This proves the claim.
\end{proof}

We have therefore proved the following proposition.
\begin{proposition}
  Maintaining the above notation:
  \begin{align*}
      \cN(3,X) &\asymp X^{1/2},\\
      \cN(6,X) &\asymp X^{1/6}\log(X),\\
      \cN(8,X) &\asymp X^{1/6}\log(X),\\
      \cN(9,X) &\asymp X^{1/6}\log(X).
  \end{align*}
\end{proposition}
\begin{proof}
    Since we only care about the 12th power of the Hodge bundle, we will take $n=12$. From Table \ref{table:ringsofmodularforms}, we observe that for the following choices of $\p$, $\cN(N,X) \asymp \cN(\p,X, \dagger)$: 
    \begin{itemize}
        \item $N=3$: $\p = (6,3,2)$,
        \item $N=6,8,9$: $\p = (6,6,6)$. 
    \end{itemize}
    The proposition now follows from Claim \ref{claim:reducingclaim} and Lemma \ref{lemma:stackycount3689}.
\end{proof}

\noindent
\textbf{The case $N=5$:} Note that this is one of the cases that cannot be tackled by the methods in \S 2 and \S 4. We first give an upper bound for $\cN(5,X)$, and then use a simple sieving argument to refine it into an asymptotic.\\

The ring of modular forms, $M(5)$ is generated by three modular forms, $C_5, \alpha_5$ and $\beta_5$ of weights 2,4 and 4 respectively. The relation between these forms is exactly:
\begin{equation}
\label{equation:X0(5)modularforms}
    O_5 = \alpha_5^2 - \beta_5(C_5^2 + 4\alpha_5 - 8\beta_5).
\end{equation}

Set $n=12$ and $\p = (6,3,3)$. Proceeding analogously as before, we must count integers $(a,b,c)$ with $\Height^{\p}(a,b,c) <X$ such that:
\begin{equation}
\label{equation:countingX0(5)}
    b^2 - a^2c - 4bc + 8c^2 = 0,
\end{equation}
and the pair $((a,b,c), \p)$ satisfies the minimality condition $(\dagger)$. If $\alpha_5 = 0$, then $\beta_5 = 0$, and we get $ \asymp X^{1/6}$ elliptic curves, which is the trivial lower bound. If $C_5 = 0$, we get the two points of $\cX_0(5)$ that have automorphism group $\mu_4$. Each of these is defined over $\Q(i)$ and doesn't contribute to the rational points on $\cX_0(5)$.\\ 

We obtain the upper bound by counting integer triples $(a,b,c)$ without the minimality condition $(\dagger)$. Equation \ref{equation:countingX0(5)} can be rearranged to one of the form:
$$
(4b - 8c)^2 + (8c - a^2)^2 = a^4
$$

For any integer $n$, let $r_2(n)$ denote the number of ways of writing an integer as a sum of two squares. An upper bound can be proved by summing $r_2(a^4)$ over all $a < X^{1/6}$.

\begin{lemma}[\cite{RecNT}, Chapter XV]
 Let $n \in \Z_{>0}$ have factorization
 $$
 n = 2^{a_0} p_1^{e_1} \ldots p_r^{e_r} q_1^{2f_1} q_2^{2f_2} \ldots q_s^{2f_s},
 $$
 where the $p_i$'s are $\equiv 1 \mod 4$ and the $q_i$'s are $\equiv 3 \mod 4$.
 Define $B(n) = \prod_{i=1}^{r}(e_i + 1)$. Then:
 \begin{align*}
     r_2(n) = 4B(n)
 \end{align*}
\end{lemma}

\begin{remark}
 This is a well known result. Note that the constant in front of $B$ is different depending on whether one takes into account signs and order. But this will not make a difference to our result, since we are only interested in the asymptotic growth rate.
\end{remark}

\noindent
We will now focus on the sum,
$$
\sum_{|a| < X^{1/6}} B^{(4)}(n),
$$
where we define $B^{(4)}(n)$ to be $B(n^4)$, for notational convenience. Note that if $p \equiv 1 \mod 4$, $B^{(4)}(p^k) = 4k+1$. If $p=2$ or $3 \mod 4$, then $B^{(4)}(p^k) = 1$ for any $k$. Thus, $B^{(4)}(n)$ is a multiplicative (although not completely multiplicative) function.


\begin{proposition}
\label{prop:X_0(5)upperbound}
    Maintaining the above notation, there is a constant $c >0$ such that for any $0< \delta <1/6$,
    $$ \sum_{|n| < X^{1/6}} B^{(4)}(n) = c X^{1/6}(\log(X))^2 + O(X^{1/6 - \delta}) .$$
\end{proposition}

\begin{proof}
    Consider the Dirichlet series: $\ds \sum_{n\ge 1} \frac{B^{(4)}(n)}{n^s}.$ By multiplicativity, this can be written as the Euler product:
    $$
    \prod_{p \equiv 1 \mod 4} \left(\sum_{k \ge 0} (4k+1)p^{-ks} \right) \prod_{p \equiv 3 \mod 4} \left(\sum_{k \ge 0}p^{-ks} \right) \left( \sum_{k \ge 0} 2^{-ks} \right).
    $$
    
    We now simplify this expression.
    \begin{align*}
        \prod_{p \equiv 3 \mod 4} \left(\sum_{k \ge 0}p^{-ks} \right) &=  \prod_{p \equiv 3 \mod 4} \frac{1}{1- p^{-s}}.
    \end{align*}
    
    \begin{align*}
        \prod_{p \equiv 1 \mod 4} \left(\sum_{k \ge 0} (4k+1)p^{-ks} \right) &= \prod_{p \equiv 1 \mod 4} \left( 4 \sum_{k \ge 0}kp^{-ks} + \sum_{k\ge 0}p^{-ks} \right) \\
        &= \prod_{p \equiv 1 \mod 4} \left( \frac{4p^{-s}}{(1-p^{-s})^2} + \frac{1}{1-p^{-s}} \right)\\
        &= \prod_{p \equiv 1 \mod 4} \left( \frac{1 + 3p^{-s}}{(1-p^{-s})^2} \right).
    \end{align*}
    
    Thus,
    \begin{align*}
        \sum_{n\ge 1}\frac{B^{(4)}(n)}{n^s} &= \prod_{p} \left( \frac{1}{1-p^{-s}} \right) \prod_{p \equiv 1 \mod 4}  \left( \frac{1 + 3p^{-s}}{1-p^{-s}} \right)\\
        &= \zeta(s) \prod_{p \equiv 1 \mod 4} \left( \frac{1 + 3p^{-s}}{1-p^{-s}} \right).
    \end{align*}
    Now, let $\chi(p)$ denote the usual Legendre Symbol $\left(\frac{-1}{p} \right)$. Let $K(s) = \frac{1-2^{-s}}{1+3.2^{-s}} $. Then,
    \begin{align*}
        \Psi(s) := \prod_{p \equiv 1 \mod 4} \left( \frac{1 + 3p^{-s}}{1-p^{-s}} \right) &= K(s)\prod_{p}  \left( \frac{1 + 3p^{-s}}{1-p^{-s}} \right)^{\frac{1+ \chi(p)}{2}}\\
        &= K(s)\prod_p \left( 1 + \frac{4p^{-s}}{1-p^{-s}}\right)^{\frac{1+ \chi(p)}{2}}\\
        &= K(s)\prod_p \left(1 + \frac{1}{2}(1+ \chi(p))\frac{4p^{-s}}{1-p^{-s}} + \ldots \right)\\
        &=K(s)  \prod_p \left(1 + 2(1+ \chi(p))p^{-s} + \text{higher powers of } p^{-s} \right).
    \end{align*}

    Consider the Dirichlet $L$-function $L(s, \chi) = \prod_p (1-\chi(p)p^{-s})^{-1}$. Since
    
    \begin{align*}
         \left(1 + 2(1+ \chi(p))p^{-s} + \ldots \right) \left(1 -\chi(p)p^{-s} \right)^2 = 1+2p^{-s} \ldots,
    \end{align*}
   
   we see that $\Psi(s)L(s,\chi)^{-2}$ has a pole of order 2 at $s=1$ and converges for $Re(s)>1$. We know that $L(s,\chi)$ is holomorphic at $s=1$. Thus $\sum_{n \ge 1} \frac{B^{(4)}(n)}{n^s} $ has a pole of order 3 at $s=1$. The proposition now follows from a standard Tauberian theorem (\cite{TschinLoir01}, Appendix A).
\end{proof}

\begin{proposition}
\label{proposition:X_0(5)final}
    For any $X>0$,
     $\cN(5,X) \asymp X^{1/6}\log(X)^2.$
\end{proposition}
\begin{proof}
     The main ingredient here is the upper bound proved in Proposition \ref{prop:X_0(5)upperbound}. To refine this to give an asymptotic growth rate, we must count only the minimal $(a,b,c)$. If a triple is \emph{non-minimal}, then there exists a prime $p$ such that $p^2|a$, $p^4|b$ and $p^4|c$. Let $p$ be such a prime. Then the number of such triples is in bijection with the number of ways of writing $a^4$ as a sum of two squares, say $a^4 = A^2 + B^2$, such that $p^4|A$ and $p^4|B$. This is the same as the number of ways of writing $(a/p^2)^4$ as a sum of two squares. Therefore the number of triples that are non-minimal at $p$ can be calculated by:
$$
\sum_{|n|<X^{1/6}/p^2} B^{(4)}(n).
$$

By Proposition \ref{prop:X_0(5)upperbound}, this has the same asymptotic growth rate as:
\begin{align*}
    c \frac{X^{1/6}}{p^2} \log \left(\frac{X}{p^{12}}\right)^2 
&= (c/p^2)X^{1/6} ( \log(X)^2 - 2\log(X)\log(p^{12}) + \log(p^{12})^2)\\
&= c X^{1/6}\log(X)^2\left(1 - \frac{1}{p^2} - 24\log(X)^{-1}\frac{\log(p)}{p^2} + 144 \log(X)^{-2}\frac{\log(p)^2}{p^2} \right),
\end{align*}
where $c$ is independent of $p$.

This leaves us to examine the product
\[
    \prod_{p^2<X^{1/6}}\left(1-\frac{1}{p^2}-24\log(X)^{-1}\frac{\log(p)}{p^2}+144\log(X)^{-2}\frac{\log(p)^2}{p^2}\right).
\]
This product is bounded both above and below by positive constants. One can see this by noting that each term is bounded below by $1-3/p^2$ and above by $1-1/p^2$. The proposition follows.

\end{proof}

\section{Open questions}

This paper raises multiple questions, some that we believe can be answered by pushing further the methods used here, and some that require different approaches. The first question is about $\cX_0(7)$. We believe that the ideas of \S 2 and \S 3 can be generalized to count points on $\cX_0(7)$, since $\cX_{1/2}(7)$ is a stacky curve with two stacky points. In this case, one must generalize Proposition \ref{prop:HarronSnowdenTwistTheorem} to the case where $f$ and $g$ are not necessarily coprime. The tricky bit here turns out to be the analogue of Lemma \ref{lemma:quadtwistlemma1}. \\ 

One might wonder whether one can count rational points on $\cX_0(7)$ via the framework in \cite{ESZB19}, as we did for some values of $N$ in \S 5. The issue with this is that for each level not listed in Table \ref{table:ringsofmodularforms}, the ring of modular forms is quite complicated. Using relations between the generators of these rings to count points on $\cX_0(N)$ can lead to very hard counting problems. For instance, the problem of counting rational points on $\cX_0(7)$ can be rephrased in terms of counting \emph{integral} points on the intersection of one cubic and two quadric hypersurfaces in $\A^5$. This gets more complicated with higher $N$, at least as far using the description in \cite{HT11} goes. For these higher $N$, if one were to find a smaller set of modular forms that could \emph{both} globally generate $\lambda^{\otimes 12}$ and had simpler relations among them, then one could perhaps count points on the corresponding $\cX_0(N)$ more easily. We do not know at this time if that is indeed possible.\\

There is of course the question of an exact asymptotic as opposed to an asymptotic growth rate. More precisely, one can ask if the limit:
$$
c_N:=\lim_{X \rightarrow \infty } \frac{\cN(N,X)}{h_N(X)}
$$
exists and what its values is. The case $N=2$ is known due to \cite{HS17}, $N=3$ due to \cite{PPV19} and $N=4$ due to \cite{PS20}. It would be interesting to calculate the values for other $N$.\\ 

\textbf{The stacky Batyrev-Manin-Malle conjecture.} For a scheme $X$ and an ample line bundle $L$ on it, the Batyrev-Manin conjecture predicts that there are constants $a(L)$ and $b(L)$ such that the number of rational points of $X$ of height bounded by a number $B$ grows like:
$$
B^{a(L)}\log(B)^{b(L)}.
$$
Here the height refers to the height with respect to the line bundle $L$. The weaker analogue states that the number of rational points should grow like $B^{a(L)+\epsilon}$. In \cite{ESZB19}, the authors make a similar conjecture for stacks, which they call the `Weak stacky Batyrev-Manin-Malle conjecture'. For each of the modular curves considered in this paper, as well as those in \cite{HS17}, the asymptotic growth rate seems to be of the same {form} as predicted, but it would be interesting to verify if the constants match the constants in \cite{ESZB19}. This is work in progress.



\appendix 

\section{Construction of $\cX_{1/2}(N)$} \label{appendix}
In this appendix, we seek to give a construction of the quotient $\cX_{1/2}(N)$ at the cusps.

\subsection{Modular description of cusps}

Let $C_n$ be a N\'eron $n$-gon. Each irreducible component of $C_n$ is isomorphic to $\P^1$. For each $i<n$, the
$i$th component is glued to the $(i+1)$st component by gluing $\infty\in\P^1_{(i)}$ to $0\in\P^1_{(i+1)}$, taking
all subscripts mod $n$. The smooth part of $C_n$, denoted $C_n^{\sm}$, is isomorphic to $\G_m\times\Z/n\Z$. The
group structure on $C_n^{\sm}$ is given by the usual group structure on each component. The automorphism group
of $C_n$ is given by $\mu_n\times\langle\operatorname{inv}\rangle$, where $\zeta\cdot(x,i)=(\zeta^ix,i)$ for
$\zeta$ a primitive $n$th root of unity and $\operatorname{inv}:(x,i)\mapsto (x\inv,-i)$. A \textit{generalized
elliptic curve $E$ over $S$} is a proper, flat, finitely presented map $E\to S$
whose geometric fibers are either smooth genus 1 curves or N\'eron $n$-gons, together with
an $S$-morphism $E^{\sm}\times_S E\to E$ which restricts to a commutative group scheme law on $E^{\sm}$. \\

Let $\barEll{n}$ denote the moduli space of generalized elliptic curves whose degenerate fibers are all
$n$-gons. In general, for any moduli stack $\cX$ of generalized elliptic curves and positive integer
$n$, let $\cX_{(n)}$ denote the substack of $\cX$ parametrizing generalized elliptic curves whose degenerate
fibers are $n$-gons.

\subsection{$\Gamma_1(N)$ and $\Gamma_0(N)$ structures }

Let $N$ be a positive integer, $n\mid N$, and $E$ a generalized elliptic curve over $S$. A \textit{$\Gamma_1(N)$
structure on $E$} is the following data:
\begin{itemize}
    \item A homomorphism $\alpha:\Z/N\Z\to E^{\sm}(S)$ such that $D:=\sum_{a\in\Z/n\Z}[\alpha(a)]$ is an
        effective Cartier divisor on $E$ forming an $S$-subgroup scheme of $E$; and
        
    \item if the fiber over some point in $S$ is an $n$-gon, then the divisor $D$ intersects every
        irreducible component of the fiber. This is equivalent to the ampleness of $D$.
\end{itemize}
The stack $\cX_1(N)$ parametrizes generalized elliptic curves with a $\Gamma_1(N)$ structure. Moreover,
we have $\cX_1(N)=\bigcup_{n\mid N}\cX_1(N)_{(n)}$. \\

Unlike the definition of a $\Gamma_1(N)$ structure, which is a fairly intuitive extension of the definition
of $\cY_1(N)$, the definition of a $\Gamma_0(N)$ structure takes more work. Let's first define a naive
$\Gamma_0(N)$ structure on a generalized elliptic curve $E/S$, and see why this will not work for us.
Our notion of a $\Gamma_0(N)$ structure will be an extension of this definition. \\

A \textit{naive $\Gamma_0(N)$ structure on E/S} is the following data:
\begin{itemize}
    \item A homomorphism $\alpha:\Z/n\Z\to E^{\sm}$ such that $D:=\sum_{a\in\Z/n\Z}[\alpha(a)]$ is an
        ample, effective divisor on $E$; and
    
    \item the image of $\alpha$ is an $S$-subgroup scheme of $E^{\sm}$.
\end{itemize}
One defines $\cX_0(N)^{\naive}$ as the moduli space of generalized elliptic curves with a naive
$\Gamma_0(N)$ structure. Again, we have $\cX_0(N)^{\naive}=\bigcup_{n\mid N}\cX_0(N)^{\naive}_{(n)}$. \\

To see why we do not use this notion, consider the modular curve $\cX_0(p^2)$ for some prime $p$.
Let $E/S$ be a generalized elliptic curve whose degenerate fiber is a $p$-gon, equipped with a naive
$\Gamma_0(p^2)$ structure $G_E$. On the degenerate fiber, the group $G$ generated by $(\zeta_{p^2},1)$
gives a naive $\Gamma_0(p^2)$ structure, andthe pair $(C_p, G)$ has automorphism group
$\mu_p\times\langle\operatorname{inv}\rangle$. On the other hand, the image of $(E,G_E)$ in $\cX_0(1)$
is a generalized elliptic curve whose degenerate fiber has automorphism group $\langle\operatorname{inv}\rangle$.
In particular, the map $\cX_0(N)^{\naive}\to\cX_0(1)$ is not representable (see Lemma 3.2.2 (b) in \cite{Ces17}).
This fails to agree with the construction of $\cX_0(N)$ given in \cite{DR73}, which is what we are
using. \\

The correct definition of a $\Gamma_0(N)$ structure is a bit long winded, so we will not
define it; rather, we will explain how to construct one from a naive $\Gamma_0(N)$ structure.
This is sufficient for our purposes. For a more detailed exposition, we refer the reader to \cite{Ces17}. Let $n\mid N$, $d(n)=n / \gcd(n, N/n)$, and $E/S$ be a
generalized elliptic curve. Let $G$ be a naive $\Gamma_0(N)$ structure on $E$, let $E^{\infty}$ denote
a degenerate fiber of $E$ which is an $n$-gon, and let $G^{\infty}$ denote the fiber of $G$ on
$E^{\infty}$. Consider the torsion subgroup $E^{\infty,\sm}[d(n)]\subset E^{\infty, \sm}$. Define
the contraction of $E$ along $E^{\sm}[d(n)]$ by leaving smooth fibers intact, and on each $E^{\infty}$
in the degenerate $n$-gon locus contracting to a point each component not intersecting $E^{\infty,\sm}[d(n)]$.
Thus, the image of $E^{\infty}$ is a $d(n)$-gon. A new elliptic curve $E'/S$ may be constructed by
gluing together the contractions of $E/S$ for each $n\mid N$ along the non-degenerate locus. The image
of $G$ under these contractions gives a $\Gamma_0(N)$ structure. Note that a $\Gamma_0(N)$ structure remembers
$G$ as well as the images of all degenerate fibers of $G$ under contractions. \\

Let $\cX_0(N)$ be the modular curve parametrizing generalized elliptic curves along with a
$\Gamma_0(N)$ structure. The following lemma shows the relationship between $\cX_0(N)^{\naive}$ and
$\cX_0(N)$.

\begin{lemma}
    There is a commutative diagram
    \[
        \begin{tikzcd}
            \cX_0(N)^{\naive}_{(n)} \ar[d] \ar[r] & \barEll{n} \ar{d} \\
            \cX_0(N)_{(d(n))} \ar{r} & \barEll{d(n)}
        \end{tikzcd}
    \]
    where the vertical maps are contractions.
\end{lemma}

\subsection{Construction of $\cX_{1/2}(N)$ at the cusps}

Recall the definition of $\cY_{1/2}(N)$ from Section \ref{section:X_1/2}. We claimed that the construction makes sense
for $\cX_{1/2}(N)$ as well, i.e.\ for generalized elliptic curves, via a similar process. We now outline
a proof of that claim. The process is the same as obtaining a $\Gamma_0(N)$ structure from a naive $\Gamma_0(N)$
structure. For the sake of clarity, let $C_n$ denote the cusp parametrizing generalized elliptic curves whose
degenerate fibers are $n$-gons. We will describe the construction on $C_n$ directly. \\

The fiber of $\cX_1(N)\to \cX_0(N)^{\naive}$ over $C_n$ consists of generators of the $\Gamma_0(N)$
structure. Thus it makes sense to define $\cX_{1/2}(N)^{\naive}$ as the fiberwise quotient of
$\cX_1(N)\to\cX_0(N)^{\naive}$ by an index 2 subgroup of $(\Z/N\Z)^{\times}$. We now declare that the fiber
of $\cX_{1/2}(N)\to\cX_0(N)$ over a cusp consists of the following data:
\begin{itemize}
    \item The fiber above the corresponding point $\cX_{1/2}(N)^{\naive}\to\cX_0(N)^{\naive}$; and
    \item the $\Gamma_0(N)$ structure on the cusp.
\end{itemize}
This second condition rigidifies the structure. \\

\begin{example}
We examine the cusps of $\cX_{1/2}(9)$. There are three possible $n$-gons in the degenerate fibers: $C_1$, $C_3$, and $C_9$.
We look at each case separately.

\begin{enumerate}
\item There is exactly one $\Gamma_0(9)$ structure on $C_1$, namely the subgroup generated by a primitive
    $9$th root of unity $\zeta_9$. The fiber of the map $\Phi_9:\cX_1(N)\to\cX_0(N)$ corresponds to the
    generators of this subgroup, namely $\{\zeta_9^i\mid i=1,2,4,5,7,8\}$. Therefore the two points in
    the fiber of $\cX_{1/2}(9)\to\cX_0(9)$ correspond to the cosets $\{\zeta_9,\zeta_9^4, \zeta_9^7\}$ and
    $\{\zeta_9^2,\zeta_9^5,\zeta_9^8\}$.
        
\item Consider the cusp $C_3$. One naive $\Gamma_0(9)$ structure on $C_3$ is generated by the pair
    $(\zeta_9, 1)$. Further, $d(3)=1$ and so $E^{\sm}[d(3)]=0$. To obtain the corresponding $\Gamma_0(9)$
    structure, one contracts the degenerate fiber to a $C_1$; the image of $\langle(\zeta_9,1)\rangle$ under
    this contraction is the subgroup generated by $(\zeta_9^3,0)$. The data of the $\Gamma_0(9)$ structure
    consists of both the original naive structure and its contraction. \\
    
    To obtain the fiber of $\cX_{1/2}(9)\to \cX_0(9)$, consider the points over $\cX_{1/2}(9)^{\naive}\to\cX_0(9)^{\naive}$.
    The fiber above $\langle(\zeta_9,1)\rangle$ corresponds to the cosets $\{(\zeta_9,1), (\zeta_9^4,1),(\zeta_9^7,1)\}$
    and $\{(\zeta_9^2,1), (\zeta_9^5,1), (\zeta_9^8,1)\}$. As an aside, note that each of these cosets
    has an automorphism group of size 3. We rigidify these points by adding in the data of the $\Gamma_0(9)$
    structure in the above definition.

\item Consider the cusp $C_9$. There is a naive $\Gamma_0(9)$ structure on $C_9$ generated by the element
    $(\zeta_9,1)$. This is also a $\Gamma_0(9)$ structure, since $d(9)=9$. The fiber of
    $\cX_{1/2}(9)^{\naive}\to\cX_0(9)^{\naive}$ thus corresponds to the cosets $\{(\zeta_9,1),(\zeta_9^4,4),(\zeta_9^7,7)\}$
    and $\{(\zeta_9^2,2),(\zeta_9^5,5),(\zeta_9^8,8)\}$.
\end{enumerate}
\end{example}


\printbibliography

\end{document}